\documentclass[11pt]{article}
\usepackage{setspace}
\usepackage{cprotect}
\usepackage{amsmath,amssymb}
\usepackage{amsthm}
\usepackage{url}
\usepackage{fullpage}
\usepackage{makeidx}
\usepackage{enumerate}
\usepackage{fullpage}
\usepackage{graphicx,float,psfrag,epsfig,caption,subcaption}
\usepackage{epstopdf}
\usepackage{amssymb,comment}
\usepackage{mathrsfs}
\usepackage{color}
\usepackage[mathscr]{euscript}
\usepackage{bbm}
\usepackage{accents}
\usepackage[style=alphabetic,maxbibnames=99]{biblatex}
\usepackage{hyperref}

\newtheorem{thm}{Theorem}[section]
\newtheorem*{thm*}{Theorem}

\newtheorem{assump}{Assumption}[section]
\newtheorem{ques}{Question}

\newtheorem{cor}{Corollary}[section]

\newtheorem{lem}[thm]{Lemma}

\theoremstyle{definition}
\newtheorem{defn}{Definition}[section]

\theoremstyle{remark}

\newcommand{\E}{\mathbb{E}}

\newcommand{\bP}{\mathbb{P}}
\newcommand{\bR}{\mathbb{R}}
\newcommand{\cB}{\mathcal{B}}
\newcommand{\cC}{\mathcal{C}}

\newcommand{\cW}{\mathcal{W}}
\newcommand{\cR}{\mathcal{R}}
\newcommand{\cI}{\mathcal{I}}

\newcommand{\cA}{\mathcal{A}}
\newcommand{\sn}{\sqrt{n}}

\newcommand{\sF}{\mathsf{F}}
\newcommand{\sL}{\mathsf{L}}
\newcommand{\sd}{\mathsf{d}}
\newcommand{\veps}{\varepsilon}

\DeclareMathOperator\sgn{sgn}

\title{Binary perceptron: efficient algorithms can find solutions in \\ a rare well-connected cluster
}
\author{Emmanuel Abbe \thanks{Institute of Mathematics, EPFL, Lausanne, CH-1015, Switzerland. Email: emmanuel.abbe@epfl.ch.} \and Shuangping Li \thanks{PACM, Princeton University, Princeton, NJ, 08544, USA. Email: sl31@princeton.edu.} \and Allan Sly \thanks{Department of Mathematics, Princeton University, Princeton, NJ, 08544, USA. Email: allansly@princeton.edu.}}
\date{}

\begin{document}
\maketitle

\begin{abstract}
It was recently shown that almost all solutions in the symmetric binary perceptron are isolated, even at low constraint densities, suggesting that finding typical solutions is hard. In contrast, some algorithms have been shown empirically to succeed in finding solutions at low density. 
This phenomenon has been justified numerically by the existence of subdominant and dense connected regions of solutions, which are accessible by simple learning algorithms. 
In this paper, we establish formally such a phenomenon for both the symmetric and asymmetric binary perceptrons. We show that at low constraint density (equivalently for overparametrized perceptrons), there exists indeed a subdominant connected cluster of solutions with almost maximal diameter, and that an efficient multiscale majority algorithm can find solutions in such a cluster with high probability, settling in particular an open problem posed by Perkins-Xu in STOC'21. In addition, even close to the critical threshold, we show that there exist clusters of linear diameter for the symmetric perceptron, as well as for the asymmetric perceptron under  additional assumptions. 

\end{abstract}

\section{Introduction}
The binary perceptron is a simple neural network model. It was studied in the 60s by Cover\footnote{Mainly for the spherical case.} \cite{cover1965geometrical} and in the 80s in the statistical physics literature with detailed characterizations put forward by Gardner and Derrida \cite{gardner1988optimal} and Krauth and M\'ezard \cite{krauth1989storage}. More recently, the structural properties of its solution space have been related to the behavior of algorithms for learning neural networks in \cite{baldassi2016unreasonable,baldassi2016local,braunstein2006learning,baldassi2015subdominant} and several probabilistic results have been established in \cite{kim1998covering,talagrand1999intersecting,stojnic2013discrete,ding2019capacity,aubin2019storage,perkins2021frozen,abbe2021proof} (see further discussions below). 

The asymmetric binary perceptron model (ABP) is defined as follows. Let $G$ be an $m$ by $n$ matrix with i.i.d.\ entries taking value in $\{+1,-1\}$ with equal probability. Fix a real number $\kappa$, and consider the following constraints:  
\begin{align*}
S_j(G):=\left\{X\in\{-1,+1\}^n: \frac{1}{\sqrt{n}}\sum_{i=1}^n
	G_{j,i} X_i
	\geq\kappa \right \}, \quad j=1,\cdots,m.
\end{align*}
We consider the regime when $m$ and $n$ go to infinity with a fixed ratio $\alpha=m/n$. The question is to characterize the structure of the following space
\begin{align*}
    S(G):=\bigcap_{j=1}^m S_j(G).
\end{align*}
In the symmetric binary perceptron (SBP), a symmetric variant of ABP originally studied in \cite{aubin2019storage}, the constraints are given by 
\begin{align*}
\tilde S_j(G):=\left\{X\in\{-1,+1\}^n: \frac{1}{\sqrt{n}}\left|\sum_{i=1}^n
	G_{j,i} X_i\right|
	\leq\kappa \right \}, \quad j=1,\cdots,m,
\end{align*}
for any $\kappa>0$, and the solution space is defined as 
\begin{align*}
    \tilde S(G):=\bigcap_{j=1}^m \tilde S_j(G).
\end{align*}

The binary perceptron model has a strong freezing property that takes place at all positive density. Namely, it was shown that most solutions are isolated in the SBP model~\cite{perkins2021frozen,abbe2021proof}, and this is also believed to be true for the ABP model~\cite{krauth1989storage,huang2013entropy}. Solutions of this form are generally expected to be hard to find as in the case of random constraint
satisfaction problems (CSPs)~\cite{krzakala2007gibbs}. However, efficient algorithms have been shown empirically to succeed in finding solutions~\cite{braunstein2006learning, baldassi2015max,baldassi2007efficient,baldassi2009generalization}, suggesting that such algorithms find atypical solutions. This phenomenon has been justified numerically by the existence of subdominant and dense connected regions of solutions~\cite{baldassi2015subdominant}. We also refer to~\cite{baldassi2021unveiling} for heuristic descriptions of the solution space. 

We establish formally such a phenomenon for both ABP and SBP. Our main result is to show that there is a large diameter cluster when the constraint density $\alpha$ is small enough. This cluster is subdominant with an exponentially small fraction of vertices for the SBP model, using in addition to the above the results from~\cite{aubin2019storage,abbe2021proof}, and we believe that the same holds for ABP. 
In addition, we also show that this wide connected cluster is accessible to polynomial time algorithms, using a multiscale majority algorithm inspired by the algorithm of Kim and Roche~\cite{kim1998covering}.

When $\alpha$ becomes larger, such a large diameter cluster no longer exists. Yet, we show that even close to the critical threshold, there exist clusters of linear diameter for the symmetric perceptron, as well as for the asymmetric perceptron under an additional assumption. This is essentially on account of different behaviours of the solutions: we show that almost all `large-margin' solutions are in linear sized clusters. 

From an algorithmic point of view, we are able to propose an efficient algorithm that finds solutions in both the ABP and the SBP model when $\alpha$ is small. This settles a open problem in \cite{perkins2021frozen}. Compared to \cite{kim1998covering}, our algorithm employs a weighted majority procedure to handle different signs that appear in the SBP model. More importantly, we actually make use of our multiscale algorithm to construct solutions indexed by a tree. Together with a scheme to interpolate solutions, we are able to establish the above structural properties of the solution space. 

\subsection{Landscape description and the wide web}

To give a heuristic visualization of the landscape of the symmetric perceptron, we picture the fitness function of solutions $-\|GX\|_\infty$ as a rugged landscape with occasional peaks at different heights as illustrated in Figure~\ref{f:mountains}.  Away from the peaks the landscape falls away quickly in most directions but more slowly in a few directions (of course unlike our figure, the space of $X$ in the perceptron model is very high dimensional).  If we fix $\alpha$, varying $\kappa$ corresponds to taking different cross-sections of the landscape as illustrated in Figure~\ref{f:cross}.  For any $\kappa$, most clusters will be isolated points given by peaks of height exactly $-\kappa$.  But the taller peaks have larger cross-sections and for small enough $\kappa$ these connect together and can form very wide but thin webs as seen in the left cross-section.  For larger $\kappa$ the mountain cross-sections do no overlap but the largest mountains still give clusters of linear size.

\begin{figure}
    \centering
    \includegraphics[width=10cm]{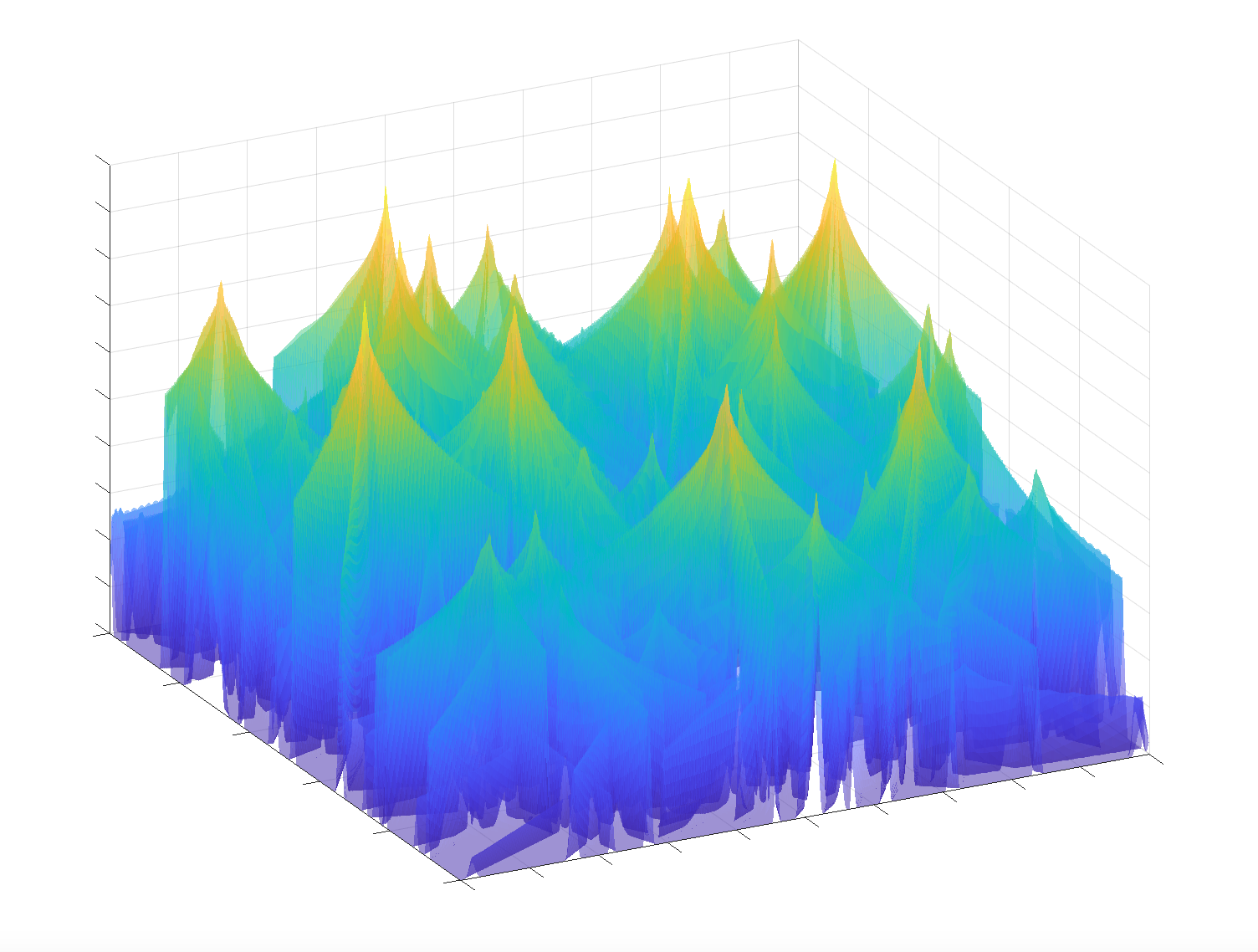}
    \caption{A heuristic illustration of the rugged landscape of the fitness function of solutions $-\|GX\|_\infty$ in the perceptron model.}
    \label{f:mountains}
    \centering
    \includegraphics[width=15cm]{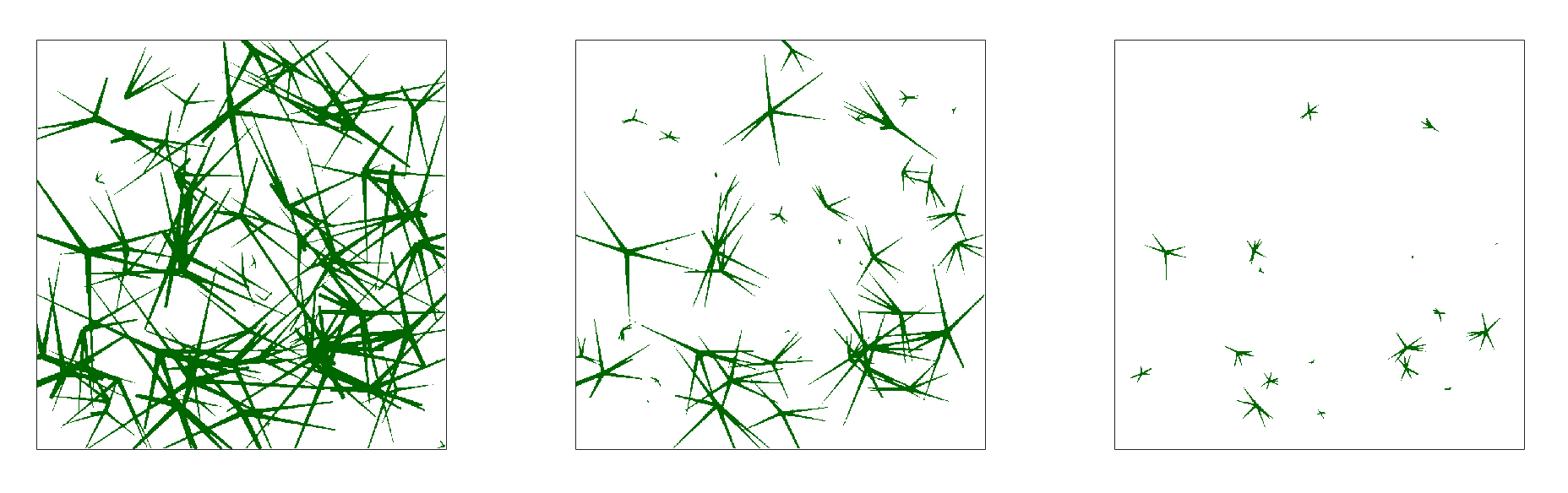}
    \caption{Cross sections $\{X:-\|GX\|_\infty \geq -\kappa\}$ at increasing values of $\kappa$ representing the clusters of solutions.}
    \label{f:cross}
\end{figure}

\subsection{The learning and CSP interplay}
Consider the problem of training a perceptron model with $\pm 1$ weights on a data distribution $P_D$ supported on $\mathcal{D} \times \mathcal{Y}$ where $\mathcal{D}=\{\pm 1\}^n, \mathcal{Y}=\{\pm 1\}$, by looking for a zero\footnote{One may be interested in a relaxed version of this where the loss is small rather than 0; the perfect interpolation regime is nonetheless of interest, due in particular to the behavior of deep learning algorithms \cite{bartlett2021deep}.} loss solution to the $(0/1)$-loss. This means that a training set $(U_i,Y_i)_{i=1}^m$ consisting of $m$ i.i.d.\ samples under $P_D$ is given, and one looks for $w \in \{\pm 1\}^n$ such that $Y_i=f_w(U_i)$ where (say for ABP)
\begin{align}
f_w(u)=\mathrm{sign}(w^{T}u/\sqrt{n}-\kappa), \quad u \in \{\pm 1\}^n.
\end{align}

In the realizable case, it is known that there exists such a function $f_w$, i.e., the data is generated consistently without noise, and the goal is to find such a function. Note that in order for the distribution of $(U_i,Y_i)$ to satisfy $Y_i=f_w(U_i)$ for some $w$, $U_i$ can be drawn uniformly at random on $\mathcal{D}$ and $Y_i$ set to $f_w(U_i)$, or $Y_i$ can be drawn uniformly at random on $\mathcal{Y}$ and $U_i$ uniformly at random under the requirement that $Y_i=f_w(U_i)$. In such cases, if there is uniqueness of the planted solution, then one has directly perfect generalization to unseen data, and without uniqueness, the implicit bias of the training algorithm (i.e., the algorithm finding an interpolator) needs to factor in to ensure  that the produced solution is a well-behaved one that generalizes. 

In general, the labels can also be assumed to be random and independent of $U_i$, in which case the training problem remains of interest from an optimization point of view. In particular, if $\kappa=0$, asking for $f_w(U_i)=Y_i$ with  uniform labels $Y_i$ independent of $U_i$ is equivalent to asking for $f_w(U_i)=1$ due to the symmetry of the model. This gives the CSP formulation for the perceptron model, where the variable $U$ corresponds to the variable $G$ and the variable $w$ corresponds to the variable $X$ from previous section. For any $\kappa$, the CSP formulation can be adapted to have the uniform label marginal, or one can simply take the labels to be always $+1$. 

In \cite{abbe2021proof}, it was shown that the planted and unplanted models are contiguous in the SAT phase for the symmetric perceptron. Thus, for any constraint density below the critical threshold, or equivalently, below the interpolation threshold of the perceptron model, the planting has little effect and exponentially many solutions exist. In this paper, we further show that for sufficiently low constraint density, or equivalently, for sufficiently overparametrized perceptron models, the solution space is not just large in size but there exists in fact a `wide web', i.e., a connected cluster of maximal diameter that must be subdominant since typical solutions are known to be isolated \cite{perkins2021frozen, abbe2021proof}. In particular, we show here that the efficient multiscale majority algorithm reaches this wide web in quadratic time. Thus,  overparametrization enables connectivity properties of the solution space that are exploitable by efficient algorithms. 

It would be interesting to further investigate whether such webs also play a role in more general settings of learning using neural networks.

\section{Results}
In this section, we describe our main results. We start with the definition of a (connected) cluster. 

\begin{defn}[Cluster] 
We say that two solutions $X_1,X_2\in S(G)$ (resp. $\tilde S(G)$ for the SBP model) are adjacent if they differ in a single coordinate. A cluster of solutions is any subset $\cC\subset S(G) \subset \{\pm 1\}^n $ that constitutes a maximal connected component of $S(G)$ under this adjacency.
\end{defn}

In the following theorem, we show the existence of a cluster with (almost) maximal diameter, i.e., the `wide web'. 
\begin{thm}\label{t:octupus}\text{}\\
1. In the SBP model, for any $\kappa>0$, there exists $\tilde\alpha_o(\kappa)>0$, such that whenever $0<\alpha<\tilde \alpha_o(\kappa)$, there exists a cluster of diameter $n$ with high probability.\\
2. In the ABP model, for any $\kappa\in \bR$ and $\veps>0$, there exists $\alpha_o(\kappa,\veps)>0$, such that whenever $0<\alpha<\alpha_o(\kappa,\veps)$, there exists a cluster of diameter at least $(1-\veps)n$ with high probability.
\end{thm}

The following theorem shows that there are efficient algorithms to locate solutions in such clusters. 
\begin{thm}\label{t:algo}\text{}\\
1. In the SBP model, for any $\kappa>0$, there exists $\tilde \alpha_a(\kappa)>0$, such that whenever $0<\alpha<\tilde \alpha_a(\kappa)$, there exists an efficient algorithm that runs in time $O(n^2)$, takes $G$ as input, and outputs a solution $X$ with high probability. Moreover, $X$ lies in a cluster of diameter $n$ with high probability.\\
2. In the ABP model, for any $\kappa \in \bR$, there exists $\alpha_a(\kappa)>0$, such that whenever $0<\alpha<\alpha_a(\kappa)$, there exists an efficient algorithm that runs in time $O(n^2)$, takes $G$ as input, and outputs a solution $X$ with high probability. Moreover, for any $\varepsilon>0$, there exists $\alpha_a(\kappa,\veps)>0$, such that whenever $0<\alpha<\alpha_a(\kappa,\veps)$, the output $X$ lies in a cluster of diameter at least $(1-\veps)n$ with high probability.
\end{thm}

Now we give the definition of a class of `large-margin' solutions.
\begin{defn}\text{}\\
1. In the SBP model, for any $\kappa'>0$, we define $X\in \{\pm 1\}^n$ to be a $\kappa'$-solution if $|GX|\leq \kappa'\sn$ holds entrywisely.\\
2. In the ABP model, for any $\kappa'\in \bR$, we define $X\in \{\pm 1\}^n$ to be a $\kappa'$-solution if $GX\geq \kappa'\sn$ holds entrywisely.
\end{defn}

With this definition, we have the following theorem.
\begin{thm}\label{t:localcluster}\text{}\\
1. In the SBP model, let $0<\kappa'<\kappa$ and $\alpha>0$. There exists $d>0$ such that $1-o_n(1)$ fraction of the $\kappa'$-solutions lie in clusters with diameter at least $d n$ with high probability.\\
2. In the ABP model, let $\kappa'>\kappa$ and $\alpha>0$. There exists $d>0$ such that $1-o_n(1)$ fraction of the $\kappa'$-solutions lie in clusters with diameter at least $d n$ with high probability.
\end{thm}

Before we present our corollary on the existence of linear sized clusters, we recall some of the previous results on the capacity threshold of binary perceptrons. 

For the SBP model, it was proven in \cite{abbe2021proof} that for any $\kappa>0$, the capacity threshold is $\tilde \alpha_c(\kappa)=-\log(2)/\log(P_\kappa)$, where $P_\kappa=\bP(|N|\leq \kappa)$, with $N$ following the standard normal distribution. More precisely, for any $\kappa>0$, $\alpha>\tilde \alpha_c(\kappa)$, and $m=\lfloor\alpha n\rfloor$, there are no solutions with high probability, whereas for any $0<\alpha<\tilde \alpha_c(\kappa)$ and $m=\lfloor\alpha n\rfloor$, there is a solution with high probability. Together with this result, we have the following corollary. 
\begin{cor}\label{c:1}
In the SBP model, for any $0<\alpha<\tilde \alpha_c(\kappa)$, there exists $d>0$ such that there is a cluster with diameter at least $dn$ with high probability.
\end{cor}

In the ABP model, our results hold conditional on the following assumption on the continuity of the threshold function.

\begin{assump}\label{a:cont}
In the ABP model, there exists a sharp threshold function $\alpha_c(\kappa)$ such that for any $\kappa>0$, $\alpha>\tilde \alpha_c(\kappa)$, and $m=\lfloor\alpha n\rfloor$, there are no solutions with high probability, whereas for any $0<\alpha<\tilde \alpha_c(\kappa)$ and $m=\lfloor\alpha n\rfloor$, there is a solution with high probability. Moreover, $\alpha_c(\kappa)$ is a right continuous function of $\kappa$.
\end{assump}

\begin{cor}\label{c:2}
In the ABP model, if Assumption \ref{a:cont} holds, then for any $0<\alpha< \alpha_c(\kappa)$, there exists $d>0$ such that there is a cluster with diameter at least $dn$ with high probability.
\end{cor}

\subsection{Open problems}
We pose a few open problems for future work.
\begin{ques}
What is the threshold for the existence of a wide web? Does it coincide with the threshold for efficient algorithms?
\end{ques}
See \cite{baldassi2015subdominant,baldassi2021unveiling} for some numerical and heuristic studies. They speculate that the threshold could be related to the local entropy being monotonic.
\begin{ques}
Is the wide web unique?
\end{ques}
When the constraint density $\alpha$ is small enough, our methods could be used to show that all solutions found by multiscale majority type algorithms are all connected to a single cluster. But it would be interesting to know whether this is the same irrespective of the algorithm and whether the wide web is indeed the unique one. Finally we can ask whether the wide web appears discontinuously.
\begin{ques}
As we lower the constraint density $\alpha$, is there a discontinuity of the asymptotic diameter of the largest-diameter cluster?
\end{ques}

\section{Majority Algorithm}
We will introduce a multiscale majority algorithm. As mentioned in the introduction, the argument is inspired by Kim and Roche \cite{kim1998covering}.

\subsection{Majority algorithm for symmetric perceptron}
We start with some definitions and notations for our algorithm. Recall that our matrix $G$ is of size $m$ by $n$. We divide the columns of $G$ into different blocks. Our algorithm is applied inductively on each block during each round. Let $R$ be the number of rounds. We will use $n_i$ and $m_i$ to denote the number of columns and rows involved in round $i$ respectively. For any $\kappa>0$, we define
\begin{align*}
&\varepsilon_0=\kappa/10, \quad C_\kappa=10/\kappa+10+\kappa, \quad \alpha_0=\kappa^4/(4C_\kappa^6), \quad m_1=m,\quad  n_1=2\lfloor (C_\kappa\sqrt{\pi mn/2})/2\rfloor+1.
\end{align*}
For $2\leq i\leq R-1$, we set
\begin{align*}
    m_{i}:=2\lfloor\Psi(i \kappa/4+5/\kappa+5) m/2\rfloor+1, \quad n_{i}:= 2\lfloor(\kappa/4) \sqrt{\pi m_i n /2}\rfloor, \quad T_{i}:=(\kappa-\varepsilon_0)(1-1/2^i)\sqrt{n}.
\end{align*}
where
\begin{align*}
    \Psi(x)=\int_{x}^\infty \frac{\exp(-u^2)}{\sqrt{2\pi}} du.
\end{align*}
We further define
\begin{align*}
    &m_R=2\lfloor\Psi(R \kappa/4+5/\kappa+5) m/2\rfloor+1, \quad 
    n_R=\lfloor (n^{0.002}\kappa/2) \sqrt{\pi m_R n /2}\rfloor, \\
    &n_0=n-\sum_{i=1}^R n_i, \quad m_0=m-\sum_{i=1}^R m_i, \quad T_0=0, \quad T_1=(\kappa-\varepsilon_0)\sqrt{n}/2.
\end{align*}
We define $R$ to be such that $m_{R-1}/m\leq n^{-0.01}<m_R/m$. 
We use $G(i:j)$ to denote the submatrix of $G$ where we keep columns from $\sum_{\ell=0}^{i-1}n_\ell+1$ to $\sum_{\ell=0}^{j}n_\ell$. And we define $G^{(r)}(i:j)$ to denote the $r$-th row of $G(i:j)$. For any vector $X$, we use $X(i:j)$ to denote the subvector of $X$ where we keep entries from $\sum_{\ell=0}^{i-1}n_\ell+1$ to $\sum_{\ell=0}^{j}n_\ell$. Further we define
\begin{align*}
    S(i:j):=G(i:j) X(i:j).
\end{align*}
This is a vector that records the corresponding inner product from round $i$ to $j$. Our algorithm will consists of a few steps. The first round and the last will be slightly different from the rest, but the general idea is similar. During each round $i$, we would look at $S(0:i-1)$ and select rows with the strongest opinions, i.e. select rows with the largest $|S^{(r)}(0:i-1)|$. When $S^{(r)}(0:i-1)$ is positive, row $r$ prefers to have $S^{(r)}(i:i)$ negative to reduce the absolute value of the total sum $S^{(r)}(0:i)$. Similarly, when $S^{(r)}(0:i-1)$ is negative, we would prefer to have a positive $S^{(r)}(i:i)$. We therefore consider a weighted majority vote procedure. For round $i$, define $\cR_i$ to be the set of $m_i$ rows with the largest $|S^{(r)}(0:i-1)|$ and define $\cC_i$ to be the index set of columns  $\{j:\sum_{\ell=0}^{i-1}n_\ell+1\leq j\leq \sum_{\ell=0}^{j}n_\ell\}$. In the following, for any positive integer $i$, we denote $[i]=\{1,2,\cdots,i\}$. And we use $\lfloor x \rfloor$ to denote the floor function.
\begin{itemize}
    \item Round 0. At round $0$, for $j\in \cC_0$, we assign $X_j$ arbitrary values.
    \item Round 1. For row $r$, define
    \begin{align*}
        \ell_r=\lfloor|S^{(r)}(0:0)|\sqrt{\pi m /2} \rfloor.
    \end{align*}
    We define a weight matrix $W$ of size $m$ by $n_1$ in the following way. For any $1\leq j\leq n_1$, define 
    \begin{align*}
        W_{r,j}:=
        \begin{cases}
         -\sgn(S^{(r)}(0:0)) G_{r,n_0+j} ,\quad &\text{if } j \leq \ell_r,\\
         -\sgn(S^{(r)}(0:0)), &\text{else}.
        \end{cases}
    \end{align*}
    Then, for any $j\in \cC_1$, define
    \begin{equation}\label{eq:major}
        X_j=\sgn\left(\sum_{r\in [m]} W_{r,j-n_0}\right).
    \end{equation}
    In short, this procedure pushes the row sum in the direction pointing to zero. Equation \eqref{eq:major} is effectively taking a majority vote of the weight in the corresponding columns.
\item Round $i$. For each round $i$ with $2\leq i\leq R-1$, for $j\in \cC_i$, define
\begin{align*}
    X_j=\sgn\left(\sum_{r\in \cR_i} -\sgn(S^{(r)}(0:i-1)) G_{r,j}\right).
\end{align*}
During each round, we push rows in $\cR_i$ to the direction of the origin by roughly a constant amount based on our design of $m_i$ and $n_i$. This is different to the original algorithm in \cite{kim1998covering}. In the SBP model, we want to keep a small step size for most rows to avoid a flipped sign and thus a larger absolute value of the new row sums. 
\item Step $R$. The last step is similar to Step 1. We define a weight matrix $\cW$ of size $m_R$ by $n_R$. For $r\in \cR_R$, define
    \begin{align*}
        \ell_r=\lfloor|S^{(r)}(0:R-1)|\sqrt{\pi m_R /2} \rfloor.
    \end{align*}
    For any $1\leq j\leq n_R$ and $r\in \cR_R$, define 
    \begin{align*}
        \cW_{r,j}:=
        \begin{cases}
         -\sgn(S^{(r)}(0:R-1)) G_{r,n-n_R+j} ,\quad &\text{if } j \leq \ell_r\\
         -\sgn(S^{(r)}(0:R-1)), &\text{else}.
        \end{cases}
    \end{align*}
    And for any $j\in \cC_R$, define
    \begin{align}\label{eq:major2}
        X_j=\sgn\left(\sum_{r\in \cR_R} \cW_{r,j+n-n_0}\right).
    \end{align}
    Similar to Step 1, this procedure pushes the row sum in the direction pointing to zero.
\end{itemize}

\subsection{Majority algorithm for asymmetric perceptron}
Notations are kept the same as in the previous section. Let $R$ be the number of rounds. We define
\begin{align*}
&\varepsilon_0=|\kappa|/10, \quad C_\kappa=10-\min(\kappa,0), \quad \alpha_0=1/(100C_\kappa^2),
\quad  m_1=m, \quad n_1=\lfloor C_\kappa\sqrt{\pi mn/2}\rfloor.
\end{align*}
For any $2\leq i\leq R-1$, define
\begin{align*}
    m_{i}:=2\lfloor\Psi(i+5) m/2\rfloor+1, \quad n_{i}:= \lfloor\sqrt{2 \pi m_i n }\rfloor, \quad T_{i}:=(\kappa+\varepsilon_0+1/2^i)\sqrt{n}.
\end{align*}
We further define
\begin{align*}
    &m_R=2\lfloor\Psi(R+5) m/2\rfloor+1, \quad 
    n_R=\lfloor n^{0.002} \sqrt{2\pi m_R n }\rfloor, \\
    &n_0=n-\sum_{i=1}^R n_i, \quad m_0=m-\sum_{i=1}^R m_i, \quad T_0=(\kappa+\varepsilon_0+1)\sqrt{n}, \quad T_1=(\kappa+\varepsilon_0+1/2)\sqrt{n}.
\end{align*}
We define $R$ to be such that $m_{R-1}/m\leq n^{-0.01}<m_R/m$. The algorithm is summarized as follows.
\begin{itemize}
    \item Round 0. At round $0$, for $j\in \cC_0$, we assign $X_j$ arbitrary values.
    \item Round $i$. For each round $i$ with $1\leq i\leq R$, for $j\in \cC_i$, define
\begin{align*}
    X_j=\sgn\left(\sum_{r\in \cR_i} G_{r,j}\right).
\end{align*}
\end{itemize}

\section{Tree indexed solutions}
\subsection{Cluster of large diameter}
We will define families of solutions indexed by a tree. Intuitively, the tree is constructed in the following way. We start with any vector $v$ of length $n_0$. Consider any vector $v^{+1}$, which we get by flipping one entry of $v$. We construct a tree where two vertices are attached to the root, each representing $v$ and $v^{+1}$. Apply our multiscale majority algorithm to $G$ and take round 0 to be $v$ and $v^{+1}$ respectively. If we denote the outputs of our algorithm in round 1 as $\cA_1(v)$ and $\cA_1(v^{+1})$, then we can consider a set of vectors of length $n_1$ that interpolate between the two vectors $\cA_1(v)$ and $\cA_1(v^{+1})$. More precisely, we start with $\cA_1(v)$, flip one entry at a time and get $\cA_1(v^{+1})$ eventually. For the tree structure, we construct a new layer of the tree by attaching all such interpolation vectors as children to $v$ and $\cA_1(v^{+1})$ as a child to $v^{+1}$. Next, we can repeat the procedure: apply round 2 of our algorithm to each branch of the tree, with step 0 equals vector in the first layer and step 1 equals vector in the second layer. Then we can further interpolate between outputs. In the end, each vertex in the tree will represent sections of solutions of the form $X(i:i)$. And each branch will be a solution $X(0:R)$. We will show that such solutions on the leaves are connected and thus establish the theorem. See Figure \ref{fig:tree} for an example. 
\begin{figure}
\includegraphics[width=\textwidth]{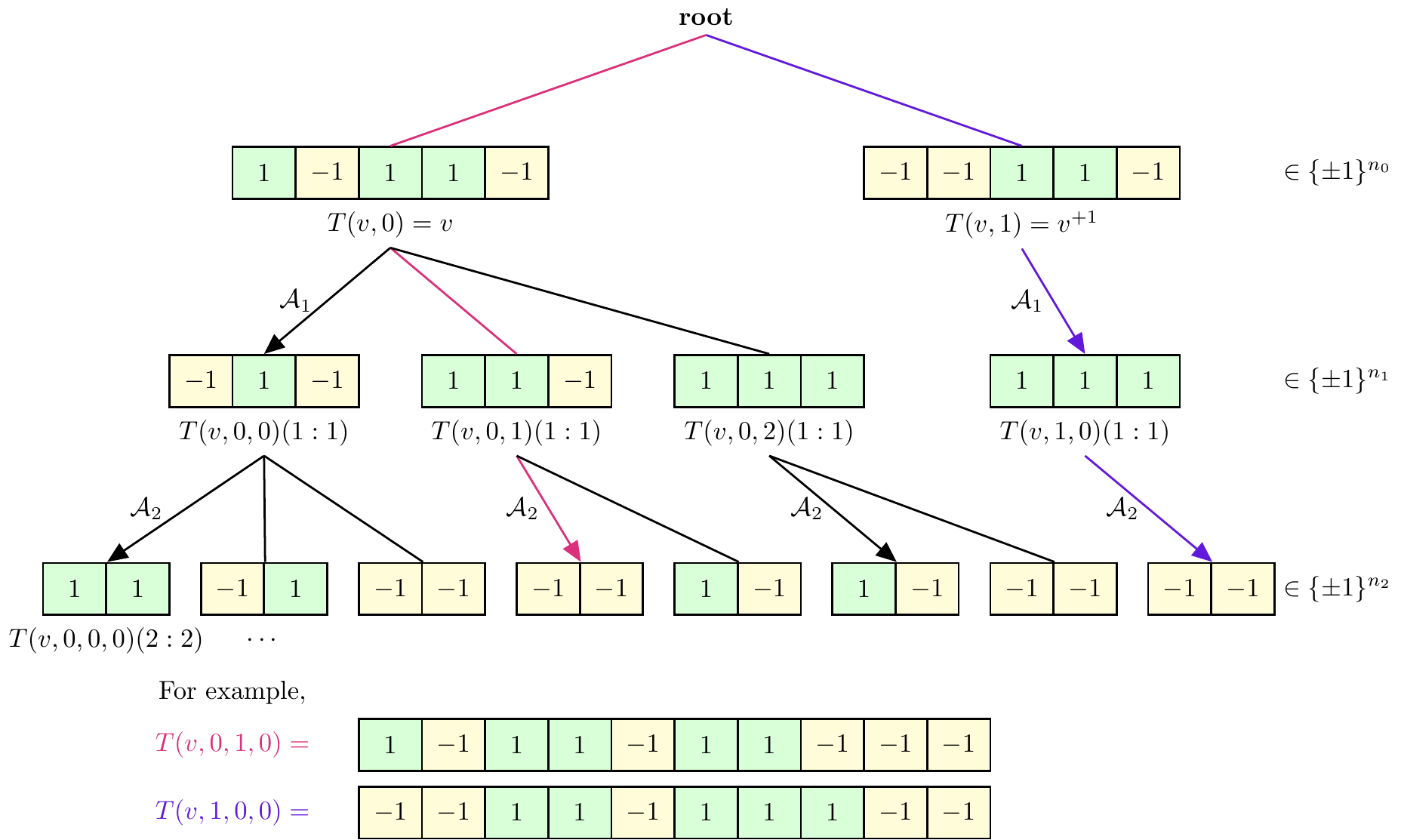}
\caption{This figure gives an example of the tree structure $T(v,h_0,\cdots,h_R)$. Here $R=2$. Each arrow with $\cA$ indicates that the segments are outputs of the algorithm. The other segments are obtained through interpolation. }
\label{fig:tree}
\end{figure}

Intuitively, the segments obtained from outputs of our algorithm are good ones. After each round, the row sums are more likely to satisfy the constraints, because of the effect of the weighted majority vote. If we look at the interpolations between two outputs, they should be relatively good as well, improving the empirical distribution of row sums at each round. Actually we can show any such vectors obtained by interpolations is a solution with probability larger than $\exp(-n^\veps)$. As we can bound the size of the tree, we can obtain the existence of clusters by a union bound. 

Now, we carry out the formal definition. For two vectors $v_1$ and $v_2$ in $\{\pm 1\}^\ell$, we define $d(v_1,v_2)$ to be their Hamming distance
and define $H(v_1,v_2):=\{j\in [\ell]: v_1(j)\neq v_2(j)\}$. Notice that $|H(v_1,v_2)|=d(v_1,v_2)$. For any $0\leq k\leq d(v_1,v_2)$, we write $H(v_1,v_2,k)$ as the set of $k$ smallest indices in $H(v_1,v_2)$. Now, for any $0\leq k\leq d(v_1,v_2)$, we define an interpolation vector $\cI(v_1,v_2,k)$ of length $\ell$ to be
\begin{align*}
    \cI(v_1,v_2,k)(i)=
    \begin{cases}
     v_2(i), \quad\text{if } i\in H(v_1,v_2,k)\\
     v_1(i), \quad\text{else }.
    \end{cases}
\end{align*}
In other words, $\cI(v_1,v_2,k)$ is an interpolation of $v_1$ and $v_2$ where we flipped $k$ entries in $v_1$ with the smallest indices. 

For $1\leq i \leq R$, we use $\cA_i$ to denote the step $i$ of our algorithm. More specifically, $\cA_i$ takes $X(0:i-1)$ as input and outputs the corresponding $X(i:i)$. Write $\cA$ as the whole algorithm, i.e. $\cA$ takes $X(0:0)$ as input and outputs $X$.

Now, we define a series of vectors of length $n$. For any vector $v \in \{\pm 1\}^{n_0}$ and $0\leq k \leq R$, we define a series of vectors $T(v,h_0,h_1,\cdots,h_k)$ inductively as follows.
\begin{itemize}
    \item When $k=0$, for $0\leq h_0\leq n_0$, define
    \begin{align*}
        T(v,h_0):=\cI(v,-v,h_0).
    \end{align*}
    For $0\leq h_0\leq n_0-1$, we define the neighbor on the right of $T(v,h_0)$ to be 
    \begin{align*}
        T^{+1}(v,h_0):=\cI(v,-v,h_0+1).
    \end{align*}
    We write $D_0:=n_0$.
    \item When $1\leq k \leq R$, define
    \begin{align*}
        D_k:=d(\cA_k(T(v,h_0,h_1,\cdots,h_{k-1})),\cA_k(T^{+1}(v,h_0,h_1,\cdots,h_{k-1}))).
    \end{align*}
    Note that actually each $D_k$ depends on $h_0, \cdots, h_{k-1}$. We suppress the dependence here for clarity. We will define the vectors $T(v,h_0,h_1,\cdots,h_k)$ for any $0\leq h_0\leq D_0-1$ and $0\leq h_i\leq D_i$ when $1\leq i\leq k$, or $h_0=D_0$ and $h_i=0$ when $1\leq i\leq k$. We understand these inequalities as if they are imposed step by step. To start with, the length of $T(v,h_0,h_1,\cdots,h_k)$ is
    \begin{align*}
        |T(v,h_0,h_1,\cdots,h_k)|=\sum_{\ell=0}^k n_\ell.
    \end{align*}
    For the first $\sum_{\ell=0}^{k-1} n_\ell$ entries, we denote them as $T(v,h_0,h_1,\cdots,h_k)(0:k-1)$ and define, 
    \begin{align*}
       T(v,h_0,h_1,\cdots,h_k)(0:k-1):=T(v,h_0,h_1,\cdots,h_{k-1}).
    \end{align*}
    Further, define the $\sum_{\ell=0}^{k-1} n_\ell+1$ to $\sum_{\ell=0}^{k} n_\ell$ entries of $T(v,h_0,h_1,\cdots,h_k)$ to be
    \begin{align*}
        T(v,h_0,h_1,\cdots,h_k)(k:k):=\cI(\cA_k(T(v,h_0,h_1,\cdots,h_{k-1})),\cA_k(T^{+1}(v,h_0,h_1,\cdots,h_{k-1})),h_k).
    \end{align*}
    And for $0\leq h_k\leq D_k-1$, we write
    \begin{align*}
        T^{+1}(v,h_0,h_1,\cdots,h_k):=T(v,h_0,h_1,\cdots,h_k+1).
    \end{align*}
    When $h_k=D_k$, we define
    \begin{align*}
        &T^{+1}(v,h_0,h_1,\cdots,h_k)(k:k):=T(v,h_0,h_1,\cdots,h_k)(k:k),\\
        &T^{+1}(v,h_0,h_1,\cdots,h_k)(0:k-1):=T^{+1}(v,h_0,h_1,\cdots,h_{k-1}).
    \end{align*}
\end{itemize}
Note that from the construction, we have that all the vectors $T(v,h_0,h_1,\cdots,h_R)$ are of length $n$ and
\begin{align*}
    d(T(v,h_0,h_1,\cdots,h_R),T^{+1}(v,h_0,h_1,\cdots,h_R))=1.
\end{align*}
Moreover, we have that
\begin{align*}
    T^{+1}(v,D_0-1,D_1,\cdots,D_R)=\cA(-v)=T(v,D_0,0,\cdots,0).
\end{align*}
In words, we defined a series of vectors $T$ that interpolate between $\cA(v)$ and $\cA(-v)$. 

In order to prove Theorem \ref{t:octupus} and Theorem \ref{t:algo}, we will show that vectors of the form $T(v,h_0,h_1,\cdots,h_R)$ are solutions with a large probability. Then by a union bound, we can have the theorem. 

\begin{defn}[Connected]
For two vectors $v_1$ and $v_2$ of length $n$, we define $v_1\sim v_2$ and say $v_1$ and $v_2$ are connected if they lie in the same cluster. 
\end{defn}
For simplicity, we use $\vec h$ to denote $(h_0,\cdots,h_R)$. We write $\vec 0\leq \vec h< \vec D$ to denote the set of inequalities $0\leq h_0\leq D_0-1$ and $0\leq h_i\leq D_i$ when $1\leq i\leq R$. And we write $\vec 0\leq \vec h\leq \vec D$ to denote $0\leq h_0\leq D_0-1$ and $0\leq h_i\leq D_i$ when $1\leq i\leq R$ or $\vec h = (D_0,0,\cdots,0)$. 
\begin{lem}\label{l:connect}
In both SBP and ABP, take $\alpha<\alpha_0$. For any $v\in \{\pm 1\}^ {n_0}$, there exists $\varepsilon>0$ such that for any $\vec 0\leq \vec h\leq \vec D$,
\begin{align*}
    \bP[T(v,h_0,h_1,\cdots,h_{R})\text{ is a solution }] \geq 1-\exp(-n^\varepsilon).
\end{align*}
\end{lem}
In other words, for any $\vec 0\leq \vec h< \vec D$, $T(v,h_0,h_1,\cdots,h_{R})$ and $T^{+1}(v,h_0,h_1,\cdots,h_{R})$ are connected with high probability. Then as each $D_i\leq n$, we have 
\begin{align*}
    \bP[\forall \vec 0\leq \vec h< \vec D, T(v,h_0,h_1,\cdots,h_R)\sim T^{+1}(v,h_0,h_1,\cdots,h_R)]\geq 1-3(n+1)^{R+1}\exp(-n^\varepsilon).
\end{align*}
Note that by design, $R=O(\sqrt{\log n})$. Therefore the right hand side converges to $1$ as $n$ goes to infinity. For SBP, note that as 
\begin{align*}
    T(v,0,\cdots,0)=\cA(v)=-\cA(-v)=-T^{+1}(v,D_0-1,D_1,\cdots,D_R),
\end{align*}
we have that $\cA(v)$ is connected to $-\cA(v)$ with high probability, and thus prove Theorem \ref{t:octupus}. Theorem \ref{t:algo} is also a direct consequence by taking all $h_i=0$. The computational complexity follows directly from construction.

For ABP, note that for any $\varepsilon>0$ and $\alpha<\min(\alpha_0,\veps^2/(10C_\kappa^2))$, the design of our algorithm guarantees that $n_0/n>1-\varepsilon$. Then,
\begin{align*}
    \cA(v)=T(v,0,\cdots,0), \quad \cA(-v)=T^{+1}(v,D_0-1,D_1,\cdots,D_{R}),
\end{align*}
with $d(\cA(v),\cA(-v))\geq (1-\varepsilon)n$. By the above computation, we have that $\cA(v)$ is connected to $\cA(-v)$ with high probability, and thus prove Theorem \ref{t:octupus}. Theorem \ref{t:algo} is also a direct consequence by taking all $h_i=0$. And the computational complexity follows directly from construction.

\subsection{Linear sized cluster}
To show the existence of linear sized clusters, we construct similar tree structures based on our multiscale majority algorithm. 

\subsubsection{Notations}
We start with a few notations. In the SBP model for any $\kappa'<\kappa$, we set the parameters of our algorithm such that
\begin{align*}
    \varepsilon_0=(\kappa-\kappa')/2, \quad \quad \varepsilon_0/\sqrt{5d}\geq C_\kappa-\log(d)\kappa/2.
\end{align*}
The rest are taken to be the same. In the ABP model for any $\kappa'>\kappa$, we set the parameters of our algorithm such that
\begin{align*}
    \varepsilon_0=(\kappa'-\kappa)/2, \quad \quad \varepsilon_0/\sqrt{5d}\geq C_\kappa-2\log(d).
\end{align*}
For a vector $v \in \{\pm 1\}^n$ and $1\leq k\leq n$, we define $\sF(v,k)$ to be a vector in $\{\pm 1\}^n$ that
\begin{align*}
    \sF(v,k)(i)=
    \begin{cases}
     -v(i), \quad&\text{if } i\leq k\\
     v(i), \quad&\text{else }.
    \end{cases}
\end{align*}
In other words, if we flip the first $k$ entries of $v$, we get $\sF(v,k)$. For a vector $v \in \{\pm 1\}^n$ and $1\leq k\leq n$, we define $v([k])$ to be the vector that consists of the first $k$ entries of $v$. More precisely, $v([k])\in \{\pm 1\}^k$ with $v([k])(i)=v(i)$ for any $1\leq i \leq k$. Note that this is different to the definition of $v(i:j)$.

Now we make a few definitions with relate to our algorithm. For a vector $v$ of length $\ell> n_0$, define $\sL(v)\in \mathbb{Z}_+$ such that
\begin{align*}
    \sum_{s=0}^{\sL(v)-1}n_s\leq \ell < \sum_{s=0}^{\sL(v)} n_s.
\end{align*}
Recall that we uses $\cA_i$ to denote the $i$-th round of our algorithm that takes a vector of length $\sum_{s=0}^{i-1} n_s$ as input and outputs a vector of length $n_i$. 
Further, we define $\cA_{0:\sL(v)}(v)$ to be a vector of length $\sum_{s=0}^{\sL(v)} n_s$ such that
\begin{align*}
    \cA_{0:\sL(v)}(v)(i):=
    \begin{cases}
     v(i), \quad&\text{if } i\leq \ell\\
     \cA_{\sL(v)}(v([\sum_{s=0}^{\sL(v)-1} n_s]))(i-\sum_{s=0}^{\sL(v)-1}n_s), \quad&\text{else }.
    \end{cases}
\end{align*}
In other words, $\cA_{0:\sL(v)}(v)$ is a vector we get by augmenting $v$ with outputs of our algorithm.

For a vector $v$ of length $\ell$, from now on in this section, we use $\cA(v)$ to denote the output of our algorithm when we take $\cA_{0:\sL(v)}(v)$ as the starting vector and round $\sL(v)+1$ as the starting round. For any $0\leq k\leq \sL(v)$, we define $\cA_{k}(v)=\cA_{0:\sL(v)}(v)(k:k)$. Similarly, for $\sL(v)+1\leq k\leq R$, we define $\cA_k(v)$ to be the vector of length $n_k$ which equals the output in round $k$ of our algorithm.

Our strategy to show connectivity is divided into two steps. Define $\sd:=\lceil dn\rceil$. If $X$ is a $\kappa'$-solution, we firstly show that $\cA(X([n-\sd]))\sim X$. Then we show that $\cA(\sF(X,\sd)([n-\sd]))\sim \cA(X([n-\sd]))$. Note the distance between $\cA(\sF(X,\sd)([n-\sd]))$ and $X$ is at least $\sd$. This guarantees that the diameter of the connected set is at least $\sd$.

\subsubsection{First Connectivity}
We construct the necessary tree structure to show $\cA(X([n-\sd]))\sim X$ with high probability. To this end, we will show that for any $n-\sd\leq \ell \leq n-1$, 
\begin{align*}
    \cA(X([\ell]))\sim \cA(X([\ell+1]))
\end{align*}
with high probability. For each $n-\sd\leq \ell \leq n-1$, we will define a series of vectors of length $n$ indexed by a tree. For any $\sL(X([\ell]))+1\leq k\leq R$, the tree is constructed by defining a series of vectors $T(X,\ell,h_{\sL(X([\ell]))+1},\cdots,h_k)$ inductively as follows. For notation purpose, we write $\sL(X([\ell]))$ as $\sL$, $\cA_{0:\sL}(X([\ell]))$ as $X_1$ and $\cA_{0:\sL}(X([\ell+1]))$ as $X_2$ for short in this subsection. Note that both $X_1$ and $X_2$ have length $\sum_{s=0}^\sL n_{s}$ and differs by at most one entry by definition. 
\begin{itemize}
    \item For $k=\sL+1$, we write $D_{k}:=d(\cA_k(X_1)),\cA_k(X_2))$. For $0\leq h_k\leq D_{k}$, $T(X,\ell,h_k)$ has length $\sum_{s=0}^k n_s$ with
    \begin{align*}
        &T(X,\ell,h_k)(0:k-1):=X_1,\\
        &T(X,\ell,h_k)(k:k):=\cI(\cA_k(X_1)),\cA_k(X_2)),h_k).
    \end{align*}
    For $0\leq h_k\leq D_k-1$, we define the neighbor on the right of $T(X,\ell,h_k)$ to be such that
    \begin{align*}
        &T^{+1}(X,\ell,h_k)(0:k-1):=T(X,\ell,h_k+1).
    \end{align*}
    And for $h_k=D_k$, we define the neighbor on the right of $T(X,\ell,h_k)$ to be such that
        \begin{align*}
        &T^{+1}(X,\ell,h_k)(0:k-1):=X_2,\\
        &T^{+1}(X,\ell,h_k)(k:k):=\cA_k(X_2).
    \end{align*}

    \item When $\sL+2\leq k \leq R$, define
    \begin{align*}
        D_k:=d(\cA_k(T(X,\ell,h_{\sL+1},\cdots,h_{k-1})),\cA_k(T^{+1}(X,\ell,h_{\sL+1},\cdots,h_{k-1}))).
    \end{align*}
    We will define $T(X,\ell,h_{\sL+1},\cdots,h_k)$
    for $0\leq h_{\sL+1}\leq D_{\sL+1}-1$ and $0\leq h_i\leq D_i$ when $\sL+2\leq i\leq k$ or $h_{\sL+1}=D_{\sL+1}$ and $h_i=0$ when $\sL+2\leq i\leq k$. To start with, the length of $T(X,\ell,h_{\sL+1},\cdots,h_k)$ is
    \begin{align*}
        |T(X,\ell,h_{\sL+1},\cdots,h_k)|=\sum_{s=0}^k n_s.
    \end{align*}
    For the first $\sum_{s=0}^{k-1} n_s$ entries, we denote them as $T(X,\ell,h_{\sL+1},\cdots,h_k)(0:k-1)$ and define, 
    \begin{align*}
       T(X,\ell,h_{\sL+1},\cdots,h_k)(0:k-1):=T(X,\ell,h_{\sL+1},\cdots,h_{k-1}).
    \end{align*}
    Further, define the $\sum_{s=0}^{k-1} n_s+1$ to $\sum_{s=0}^{k} n_s$ entries of $T(X,\ell,h_{\sL+1},\cdots,h_k)$ to be
    \begin{align*}
        T(X,\ell,h_{\sL+1},\cdots,h_k)(k:k):=\cI(&\cA_k(T(X,\ell,h_{\sL+1},\cdots,h_{k-1})),\\
        &\cA_k(T^{+1}(X,\ell,h_{\sL+1},\cdots,h_{k-1})),h_k).
    \end{align*}
    And for $0\leq h_k\leq D_k-1$, we write
    \begin{align*}
        T^{+1}(X,\ell,h_{\sL+1},\cdots,h_k):=T(X,\ell,h_{\sL+1},\cdots,h_k+1).
    \end{align*}
    When $h_k=D_k$, we define
    \begin{align*}
        &T^{+1}(X,\ell,h_{\sL+1},\cdots,h_k)(k:k):=T(X,\ell,h_{\sL+1},\cdots,h_k)(k:k),\\
        &T^{+1}(X,\ell,h_{\sL+1},\cdots,h_k)(0:k-1):=T^{+1}(X,\ell,h_{\sL+1},\cdots,h_{k-1}).
    \end{align*}
\end{itemize}
Similar to the previous subsection, we use $\vec h$ to denote $(h_{\sL+1},\cdots,h_R)$. We write $\vec 0\leq \vec h< \vec D$ to denote the set of inequalities $0\leq h_{\sL+1}\leq D_{\sL+1}-1$ and $0\leq h_i\leq D_i$ when $\sL+2\leq i\leq R$. And we write $\vec 0\leq \vec h\leq \vec D$ to denote all $\vec h$ such that $0\leq h_{\sL+1}\leq D_{\sL+1}-1$ and $0\leq h_i\leq D_i$ when $\sL+2\leq i\leq R$ or $\vec h = (D_{\sL+1},0,\cdots,0)$. 
\begin{lem}\label{l:connect2}
Let $0<\kappa'<\kappa$ (in the SBP model) or $\kappa'>\kappa$ (in the ABP model). For any $\kappa'$-solution $X$, there exists $d>0$ and $\veps>0$ such that for any $n-\sd\leq \ell \leq n-1$ and $\vec 0\leq \vec h\leq \vec D$,
\begin{align*}
    \bP[T(X,\ell,h_{\sL+1},\cdots,h_{R}) \text{ is a solution }] \geq 1-\exp(-n^\varepsilon).
\end{align*}
\end{lem}
In other words, for any $\vec 0\leq \vec h< \vec D$, $T(X,\ell,h_{\sL+1},\cdots,h_{R})$ and $T^{+1}(X,\ell,h_{\sL+1},\cdots,h_{R})$ are connected with high probability. Then as each $D_i\leq n$, we have 
\begin{align*}
    \bP[\forall \vec 0\leq \vec h< \vec D, T(X,\ell,h_{\sL+1},\cdots,h_{R})\sim T^{+1}(X,\ell,h_{\sL+1},\cdots,h_{R})]\geq 1-3(n+1)^{R+1}\exp(-n^\varepsilon).
\end{align*}
Note that as 
\begin{align*}
    &T(X,\ell,0,\cdots,0)=\cA(X([\ell])),\\
    &T^{+1}(X,\ell,D_{\sL+1}-1,D_{\sL+2},\cdots,D_R)=\cA(X([\ell+1])),
\end{align*}
$\cA(X([\ell]))$ is connected to $\cA(X([\ell+1]))$ with probability larger than $1-3(n+1)^{R+1}\exp(-n^\varepsilon)$. Therefore, $\cA(X([n-\sd]))$ is connected to $X$ with high probability. 

\subsubsection{Second Connectivity}
Similar to the previous section, we construct a tree structure to show $\cA(\sF(X,\sd)([n-\sd]))\sim \cA(X([n-\sd]))$. To this end, we will show that for any $0\leq \ell\leq \sd-1$, 
\begin{align*}
    \cA(\sF(X,\ell)([n-\sd]))\sim \cA(\sF(X,\ell+1)([n-\sd]))
\end{align*}
with high probability. For each $0\leq \ell\leq \sd-1$, we will define a series of vectors of length $n$ indexed by a tree. For any $\sL(X[n-\sd])\leq k\leq R$, the tree is constructed by defining a series of vectors $T_2(X,\ell, h_{\sL(X([n-\sd]))},\cdots,h_k)$ inductively as follows. For notation purpose, we write $\sL(X([n-\sd]))$ as $\sL$, $\sF(X,\ell)([n-\sd])$ as $X_1$ and $\sF(X,\ell+1)([n-\sd])$ as $X_2$. Note that $X_1$ and $X_2$ both have length $n-\sd$ and they differ by at most one entry.
\begin{itemize}
    \item For $k=\sL$, we write $D_{k}:=d(\cA_k(X_1),\cA_k(X_2))$. For $0\leq h_k\leq D_{k}$, $T_2(X,\ell,h_k)$ has length $\sum_{s=0}^\sL n_s$ with
    \begin{align*}
        &T_2(X,\ell,h_k)(0:k-1):=X_1(0:k-1),\\
        &T_2(X,\ell,h_k)(k:k):=\cI(\cA_k(X_1)),\cA_k(X_2)),h_k).
    \end{align*}
    For $0\leq h_k\leq D_k-1$, we define the neighbor on the right of $T_2(X,\ell,h_k)$ to be such that
    \begin{align*}
        T_2^{+1}(X,\ell,h_k):=T_2(X,\ell,h_k+1)
    \end{align*}
    And for $h_k=D_k$, we define the neighbor on the right of $T_2(X,\ell,h_k)$ to be such that
        \begin{align*}
        &T_2^{+1}(X,\ell,h_k)(0:k-1):=X_2(0:k-1),\\
        &T_2^{+1}(X,\ell,h_k)(k:k):=\cA_k(X_2).
    \end{align*}

    \item When $\sL+1\leq k \leq R$, define
    \begin{align*}
        D_k:=d(\cA_k(T_2(X,\ell,h_{\sL},\cdots,h_{k-1})),\cA_k(T_2^{+1}(X,\ell,h_{\sL},\cdots,h_{k-1}))).
    \end{align*}
    We will define $T_2(X,\ell,h_{\sL},\cdots,h_k)$ for
    $0\leq h_\sL\leq D_\sL-1$ and $0\leq h_i\leq D_i$ when $\sL+1\leq i\leq k$, or $h_\sL=D_\sL$ and $h_i=0$ when $\sL+1\leq i\leq k$. To start with, the length of $T_2(X,\ell,h_{\sL},\cdots,h_k)$ is
    \begin{align*}
        |T_2(X,\ell,h_{\sL},\cdots,h_k)|=\sum_{s=0}^k n_s.
    \end{align*}
    For the first $\sum_{s=0}^{k-1} n_s$ entries, we denote them as $T_2(X,\ell,h_{\sL},\cdots,h_k)(0:k-1)$ and define, 
    \begin{align*}
       T_2(X,\ell,h_{\sL},\cdots,h_k)(0:k-1):=T_2(X,\ell,h_{\sL},\cdots,h_{k-1}).
    \end{align*}
    Further, define the $\sum_{s=0}^{k-1} n_s+1$ to $\sum_{s=0}^{k} n_s$ entries of $T_2(X,\ell,h_{\sL},\cdots,h_k)$ to be
    \begin{align*}
        T_2(X,\ell,h_{\sL},\cdots,h_k)(k:k):=\cI(\cA_k(T_2(X,\ell,h_{\sL},\cdots,h_{k-1})),\cA_k(T_2^{+1}(X,\ell,h_{\sL},\cdots,h_{k-1})),h_k).
    \end{align*}
    And for $0\leq h_k\leq D_k-1$, we write
    \begin{align*}
        T_2^{+1}(X,\ell,h_{\sL},\cdots,h_k):=T_2(X,\ell,h_{\sL},\cdots,h_k+1).
    \end{align*}
    When $h_k=D_k$, we define
    \begin{align*}
        &T_2^{+1}(X,\ell,h_{\sL},\cdots,h_k)(k:k):=T_2(X,\ell,h_{\sL},\cdots,h_k)(k:k),\\
        &T_2^{+1}(X,\ell,h_{\sL},\cdots,h_k)(0:k-1):=T_2^{+1}(X,\ell,h_{\sL},\cdots,h_{k-1}).
    \end{align*}
\end{itemize}

Similar to the previous subsection, we use $\vec h$ to denote $(h_{\sL},\cdots,h_R)$. We write $\vec 0\leq \vec h< \vec D$ to denote the set of inequalities $0\leq h_{\sL}\leq D_{\sL}-1$ and $0\leq h_i\leq D_i$ when $\sL+1\leq i\leq R$. And we write $\vec 0\leq \vec h\leq \vec D$ to denote all $\vec h$ such that $0\leq h_{\sL}\leq D_{\sL}-1$ and $0\leq h_i\leq D_i$ when $\sL+1\leq i\leq R$ or $\vec h = (D_{\sL},0,\cdots,0)$. 
\begin{lem}\label{l:connect3}
Let $0<\kappa'<\kappa$ (in the SBP model) or $\kappa'>\kappa$ (in the ABP model). For any $\kappa'$-solution $X$, there exists $d>0$ and $\veps>0$ such that for any $0\leq \ell \leq \sd$ and $\vec 0\leq \vec h\leq \vec D$,
\begin{align*}
    \bP[T_2(X,\ell,h_{\sL},\cdots,h_{R}) \text{ is a solution }] \geq 1-\exp(-n^\varepsilon).
\end{align*}
\end{lem}
In other words, for any $\vec 0\leq \vec h< \vec D$, $T_2(X,\ell,h_{\sL},\cdots,h_{R})$ and $T_2^{+1}(X,\ell,h_{\sL},\cdots,h_{R})$ are connected with high probability. Then as each $D_i\leq n$, we have 
\begin{align*}
    \bP[\forall \vec 0\leq \vec h< \vec D, T_2(X,\ell,h_{\sL},\cdots,h_{R})\sim T_2^{+1}(X,\ell,h_{\sL},\cdots,h_{R})]\geq 1-3(n+1)^{R+1}\exp(-n^\varepsilon).
\end{align*}
Note that as 
\begin{align*}
    &T_2(X,\ell,0,\cdots,0)=\cA(\sF(X,\ell)([n-\sd])),\\
    &T_2^{+1}(X,\ell,D_{\sL}-1,D_{\sL+1},\cdots,D_R)=\cA(\sF(X,\ell+1)([n-\sd])),
\end{align*}
we have that $\cA(\sF(X,\ell)([n-\sd]))$ is connected to $\cA(\sF(X,\ell+1)([n-\sd]))$ with probability larger than $1-3(n+1)^{R+1}\exp(-n^\varepsilon)$. Therefore, $\cA(\sF(X,\sd)([n-\sd]))$ is connected to $\cA(X([n-\sd]))$ with high probability. Together with the first connectivity, we have Theorem \ref{t:localcluster}.

\section{Proof of Theorem \ref{t:octupus} and Theorem \ref{t:algo}}

\subsection{Inductive bound on distribution in SBP model}\label{ss:inductSBP}
In this section, we consider the SBP model. We control the empirical distribution of the inner product $G(0:s)T(h_0,h_1,\cdots,h_s)$ for any $s\leq R$. Our proof is partly adapted from \cite{kim1998covering}. Many notations and results are similar, so we only present here the key arguments.

We start with a few notations. Recall that for any vector $v$, we use $v(i:j)$ to denote a section of $v$ where we keep entries from $\sum_{\ell=0}^{i-1}n_\ell +1$ to $\sum_{\ell=0}^j n_\ell$. Moreover, we define the inner product
\begin{align*}
    S(v,h_0,h_1,\cdots,h_s):= G(0:s)T(v,h_0,h_1,\cdots,h_s).
\end{align*}
The following lemma will be crucial in proving Lemma \ref{l:connect}. 
\begin{lem}\label{l:together2}
For any $1\leq k \leq R-1$, $\eta\geq T_k$, and for any $\vec 0\leq \vec h \leq \vec D$,
\begin{align*}
    \bP_2[|S(v,h_0,h_1,\cdots,h_k)|\geq \eta] \leq 2\cdot 3^k\left( n^{-1/10} +\Psi \left( \frac{\eta+\mu_{k,1}}{\sigma_{k,1}}\right)+(3/2)(1+n^{-1/10})^k \Psi \left( \frac{\eta+\mu_{k,2}}{\sigma_{k,2}}\right)\right),
\end{align*}
with probability larger than $1-\exp(-n^{1/80})$. Here, 
\begin{align*}
    \mu_{k,1}=-T_{k-1}, \quad \sigma_{k,1}^2=n_k, \quad \mu_{k,2}=\sum_{i=1}^k n_i\sqrt{2/\pi m_i}, \quad \sigma_{k,2}^2=\sum_{i=0}^k n_i.
\end{align*}
Furthermore, the right hand side is bounded from above by $m_{k+1}/m$ for any $1\leq k\leq R-1$.
\end{lem}
In the rest of the section, we provide some important lemmas that are necessary to prove Lemma \ref{l:together2}. Some of their proofs are deferred to later sections.
\begin{defn}
We define for $k=1,\cdots,R$,
\begin{align*}
    \lambda_k:=(1-m_k^{-1/8})\sqrt{2/(\pi m_k)}, \quad \bar \lambda_k:=(1+m_k^{-1/8})\sqrt{2/(\pi m_k)}.
\end{align*}
\end{defn}
The first lemma controls the behaviour of majority vote. We take the convention that if $X\sim \mathrm{Bern}(p)$, then $X=1$ with probability $p$ and $-1$ with probability $1-p$. 
\begin{lem}\label{l:ulbound}
Let $b\leq t^{1/10}$. Let $\xi_1,\cdots,\xi_t$ be i.i.d. Rademacher random variables. Let $\cB=\{m_1,\cdots,m_b\}$ be any $b$-element subset of $\{1,\cdots,t\}$. Then we can define mutually i.i.d. random variables $\psi_1,\cdots,\psi_b \sim \mathrm{Bern}(\frac{1}{2}(1+(1-t^{-1/8})\sqrt{2/(\pi t )}))$ on the same space as the $\xi$'s such that
\begin{align*}
    \psi_j\leq \xi_{m_j} \sgn \left ( \sum_{i=1}^t \xi_i\right ),
\end{align*}
for any large enough $t$. Similarly, we can define mutually i.i.d. random variables $\bar \psi_1,\cdots, \bar \psi_b \sim \mathrm{Bern}(\frac{1}{2}(1+(1+t^{-1/8})\sqrt{2/(\pi t )}))$ on the same space as the $\xi$'s such that
\begin{align*}
    \bar \psi_j \geq \xi_{m_j} \sgn \left ( \sum_{i=1}^t \xi_i\right ),
\end{align*}
for any large enough $t$.
\end{lem}
\begin{proof}
The first half of the lemma is the same as Lemma 5.1 in \cite{kim1998covering}. Proof of the second half is the same.
\end{proof}
The next lemma gives control for binomial random variables. 
\begin{lem}[Lemma 5.4 in \cite{kim1998covering}]\label{l:binomial0}
For any $\varepsilon>0$. For any $p\in [1/10,9/10]$, and $m$ large enough, we have
\begin{align*}
    &\bP(B\leq \eta)\leq \exp(-m^{1/6})+(1+\varepsilon)\Psi\left(\frac{-\eta+mp}{(mp(1-p))^{1/2}}\right),\\
    &\bP(B\geq \eta)\leq \exp(-m^{1/6})+(1+\varepsilon)\Psi\left(\frac{\eta-mp}{(mp(1-p))^{1/2}}\right).
\end{align*}
 where $B\sim \mathrm{Bin}(m,p)$.
\end{lem}
The following lemma gives bound for the probability distribution in round 1. 
\begin{lem}\label{l:round1sum}
Write $t=m_1$ and assume $\{r_1,\cdots,r_t\}$ are the $t$ elements of $\cR_1$. For $\eta>0$, we define $\gamma_i=I(|S^{(r_i)}(0:1)|\geq \eta)$ and 
\begin{align*}
    q:=\exp(-m^{1/5})+\frac{7}{3}\Psi\left( \frac{\eta}{\sqrt{n_1}  }\right)+\frac{7}{3}\Psi\left( \frac{\eta+n_1\sqrt{2/(\pi m)}}{\sqrt{n_0+n_1}  }\right).
\end{align*}
Then for large enough $m$, we have
\begin{align*}
    \bP \left( \left|\sum_{i=1}^t \gamma_i \right |\geq t^{3/5}+(1+t^{-1/{12}})qt  \right)\leq \exp(-t^{1/{70}}).
\end{align*}
\end{lem}

The following lemma gives bound on the empirical distribution. We use $\bP_2$ to denote the empirical distribution. For any set $\cR$ of rows, we use $\bP_2(S(i:j\mid \cR)\leq u)$ to denote the empirical cumulative distribution condition on $r\in \cR$. More precisely, $\bP_2(S(i:j\mid \cR)\leq u)=|\{r\in \cR:S^{(r)}(i:j)\leq u\}|/|\cR|$.
\begin{lem}\label{l:step1p2}
With probability larger than $1-\exp(-n^{1/71})$, for any $\eta\geq 0$, and $n$ large enough, we have
\begin{align}\label{eq:inductbound0}
    \bP_2(|S(0:1)|\geq \eta)\leq (9/4)n_1^{-2/5}+(5/2)\Psi\left ( \frac{\eta+\mu_{1,1}}{\sigma_{1,1}}\right )+(5/2)\Psi\left ( \frac{\eta+\mu_{1,2}}{\sigma_{1,2}}\right ).
\end{align}
\end{lem}
\begin{proof}
By Lemma \ref{l:round1sum} and the union bound, we have that for any $\eta\geq 0$, and large enough $n$,
\begin{align*}
    \bP_2(|S(0:1)|\geq \eta) &\leq t^{-2/5}+(1+t^{-1/{12}}) \left( \exp(-m^{1/5})+\frac{7}{3}\Psi\left( \frac{\eta}{\sqrt{n_1}  }\right)+\frac{7}{3}\Psi\left( \frac{\eta+n_1\sqrt{2/(\pi m)}}{\sqrt{n_0+n_1}  }\right)\right)\\
    &\leq (9/4)n_1^{-2/5}+(5/2)\Psi\left ( \frac{\eta+\mu_{1,1}}{\sigma_{1,1}}\right )+(5/2)\Psi\left ( \frac{\eta+\mu_{1,2}}{\sigma_{1,2}}\right ),
\end{align*}
with probability larger than $1-\exp(-n^{1/71})$.
\end{proof}

The next few lemmas will be on the general induction steps. We firstly record a lemma that controls exchangeable $0-1$ random variables. 
\begin{lem}[Lemma 5.5 in \cite{kim1998covering}]\label{l:KRexchange}
Suppose that $\xi_1,\cdots,\xi_t$ are exchangeable $0-1$ random variables. Let $b:=2\lfloor(1/2)t^{1/10}\rfloor$, and suppose that 
\begin{align*}
    \bP(\xi_1=\cdots=\xi_b=1)\leq q^b,
\end{align*}
where $q\in (0,1)$. Then there exists an absolute constant $t_0$ such that for all $t\geq t_0$,
\begin{align*}
    \bP\left(\sum_{i=1}^t \xi_i\geq t^{3/5}+(1+t^{-1/12})qt\right)\leq \exp(-t^{1/70}).
\end{align*}
\end{lem}
With this lemma, we have the following. 
\begin{lem}\label{l:roundisum}
For any $1<k\leq R$, write $t=m_k$ and assume $\{r_1,\cdots,r_t\}$ are the $t$ elements of $\cR_k$. For $\eta \in \bR$, we define 
\begin{align*}
    &\gamma_i:=I(\sgn(S^{(r_i)}(0:k-1))S^{(r_i)}(k:k)\geq \eta),\\
    &\bar \gamma_i:=I(\sgn(S^{(r_i)}(0:k-1))S^{(r_i)}(k:k)\leq \eta).
\end{align*}
Further, define
\begin{align*}
    &q:=\exp(-m_k^{1/6})+\frac{9}{8}\Psi\left( \frac{\eta+ \lambda_k n_k}{ (n_k(1-\lambda_k^2))^{1/2} }\right),\\
    &\bar q:=\exp(-m_k^{1/6})+\frac{9}{8}\Psi\left( \frac{-\eta- \bar \lambda_k n_k}{ (n_k(1- \bar\lambda_k^2))^{1/2} }\right).
\end{align*}
 Then we have for $n$ large enough,
\begin{align*}
    \bP \left( \sum_{i=1}^t \gamma_i \geq t^{3/5}+(1+t^{-1/{12}})qt  \right)\leq \exp(-t^{1/{70}}),
\end{align*}
and
\begin{align*}
    \bP \left( \sum_{i=1}^t \bar \gamma_i \geq t^{3/5}+(1+t^{-1/{12}})\bar qt  \right)\leq \exp(-t^{1/{70}}).
\end{align*}
\end{lem}
\begin{proof}
Note that for each $r\in \cR_k$,
\begin{align*}
    S^{(r)}(k:k)&= -\sum_{j\in \cC_k}\sgn(\sum_{r\in \cR_k} \sgn(S^{(r)}(0:k-1)) G_{r,j})) G_{r,j}\\
    &= -\sgn(S^{(r)}(0:k-1))\sum_{j\in \cC_k}G_{r,j}\sgn(\sum_{r\in \cR_k} G_{r,j}).
\end{align*}
For any $\cB=\{B_1,\cdots,B_b\}$ $b$-element subset of $\cR_k$, for any $r\in \cB$, by Lemma \ref{l:ulbound}, we can define i.i.d. $\psi_{j,r}\sim \mathrm{Bern}(\frac{1}{2}(1+(1-m_k^{-1/8})\sqrt{2/(\pi m_k )}))$ and i.i.d. $\bar \psi_{j,r} \sim  \mathrm{Bern}(\frac{1}{2}(1+(1+m_k^{-1/8})\sqrt{2/(\pi m_k )}))$ such that $\psi_{j,r}\leq G_{r,j} \sgn(\sum_{r\in \cR_k} G_{r,j}) \leq \bar \psi_{j,r}$ and thus
\begin{align}\label{eq:srbound1}
    - \sum_{j\in \cC_k}^{\ell_r} \bar \psi_{j,r}\leq \sgn(S^{(r)}(0:k-1))S^{(r)}(k:k)\leq-  \sum_{j\in \cC_k}^{\ell_r}  \psi_{j,r}.
\end{align}
By Lemma \ref{l:binomial0}, we have that for any $\eta\in \mathbb{R}$,
\begin{align*}
    \bP(\sgn(S^{(r)}(0:k-1))S^{(r)}(k:k)\geq \eta)\leq \exp(-m_k^{1/6})+\frac{9}{8}\Psi\left( \frac{\eta+ \lambda_k n_k}{ (n_k(1-\lambda_k^2))^{1/2} }\right),
\end{align*}
and
\begin{align*}
    \bP(\sgn(S^{(r)}(0:k-1))S^{(r)}(k:k)\leq \eta)\leq \exp(-m_k^{1/6})+\frac{9}{8}\Psi\left( \frac{-\eta- \bar \lambda_k n_k}{ (n_k(1-\bar\lambda_k^2))^{1/2} }\right).
\end{align*}
As these hold for any $r \in \cB$, by Lemma \ref{l:KRexchange}, we have the lemma. 
\end{proof}

\begin{lem}
With probability larger than $1-\exp(-n^{1/71})$, for any $1<k\leq R$ and $n$ large enough, we have
\begin{align*}
    \bP_2(\sgn(S(0:k-1\mid \cR_k))S(k:k\mid \cR_k)\geq \eta)\leq (9/8)m_k^{-2/5}+(5/4)\Psi\left ( \frac{\eta+\lambda_k n_k}{(n_k(1-\lambda_k^2))^{1/2}}\right ),
\end{align*}
and 
\begin{align*}
    \bP_2(\sgn((0:k-1\mid \cR_k))S(k:k\mid \cR_k)\leq \eta)\leq (9/8)m_k^{-2/5}+(5/4)\Psi\left ( \frac{-\eta-\bar \lambda_k n_k}{(n_k(1-\bar \lambda_k^2))^{1/2}}\right ).
\end{align*}
\end{lem}
\begin{proof}
By Lemma \ref{l:roundisum} and the union bound.
\end{proof}

We now state a lemma that gives a bound on the statistical distribution based on sample distribution. This will be useful later in the induction.

\begin{lem}[Lemma 5.7 in \cite{kim1998covering}]\label{l:sam2stat}
For any real numbers $c_1\leq \cdots \leq c_t$, we use $\hat F$ to denote their cdf. Define mutually iid random variables $\xi_1,\cdots,\xi_t$ with statistical cdf
\begin{align*}
    H(u)=\begin{cases}
     1 &\quad \text{ if } u\geq c_t\\
     \max\{-t^{-1/6}+\hat F(u),0\} &\quad \text{ if } u< c_t,
    \end{cases}
\end{align*}
and let $\hat H$ be the sample cdf of $\xi$'s. Then for $t$ large enough,
\begin{align*}
    \bP(\hat H(u)\leq \hat F(u), \forall u \in \bR)\geq 1-\exp(-t^{1/9}).
\end{align*}
Similarly, define mutually iid random variables $\bar\xi_1,\cdots,\bar\xi_t$ with statistical cdf
\begin{align*}
    H(u)=\begin{cases}
     \min\{t^{-1/6}+\hat F(u),1\} &\quad \text{ if } u\geq c_1\\
     0 &\quad \text{ if } u< c_1,
    \end{cases}
\end{align*}
and let $\hat H$ be the sample cdf of $\bar \xi$'s. Then for $t$ large enough,
\begin{align*}
    \bP(\hat H(u)\geq \hat F(u), \forall u \in \bR)\geq 1-\exp(-t^{1/9}).
\end{align*}
\end{lem}
And we further have the following lemma.
\begin{lem}[Lemma 5.8 in \cite{kim1998covering}]\label{l:sam2stat2}
Let $U_1,\cdots,U_t$ be exchangeable real valued random variables taking values on a finite set. Suppose with probability larger than $1-\exp(-t^{1/73})$, their sample cdf is bounded from below by $\hat F$. Then we define iid random variables $W_1,\cdots,W_t$ on the space as $U_i$'s with cdf $H$ as in Lemma \ref{l:sam2stat}, such that with probability larger than $1-\exp(-t^{1/73})$,
\begin{align*}
    U_i\leq W_i \quad \forall i \in \{1,\cdots,t\}.
\end{align*}
Similarly, suppose with probability larger than $1-\exp(-t^{1/73})$, their sample cdf is dominated by $\hat F$. Then we define iid random variables $\bar W_1,\cdots,\bar W_t$ on the space as $U_i$'s with cdf $H$ as in Lemma \ref{l:sam2stat}, such that with probability larger than $1-\exp(-t^{1/73})$,
\begin{align*}
    U_i\geq \bar W_i \quad \forall i \in \{1,\cdots,t\}.
\end{align*}
\end{lem}

For the other direction, we now state a lemma that gives a bound on the sample distribution based on statistical distribution.
\begin{lem}[Corollary 5.1 in \cite{kim1998covering}]\label{l:stat2sam}
Suppose that $\psi_1,\cdots,\psi_t$ are iid random variables with common cdf $F$, with all $\psi$'s drawn from a set of size at most $M$. Suppose $F(u)\leq H(u)$ for any $u\in \bR$. Let $\hat F$ be the sample cdf of $\psi$'s. Then when $t$ is large enough,
\begin{align*}
    \bP(\hat F(u)\leq t^{-2/5}+(1+t^{-1/12})H(u), \forall u \in \bR)\geq 1-M\exp(-t^{1/70}).
\end{align*}
Similarly, for the other direction, suppose that $\psi_1,\cdots,\psi_t$ are iid random variables with common cdf $F$, with all $\psi$'s drawn from a set of size at most $M$. Suppose $F(u)\geq H(u)$ for any $u\in \bR$. Let $\hat F$ be the sample cdf of $\psi$'s. Then when $t$ is large enough,
\begin{align*}
    \bP(1-\hat F(u)\leq t^{-2/5}+(1+t^{-1/12})(1-H(u)), \forall u \in \bR)\geq 1-M\exp(-t^{1/70}).
\end{align*}
\end{lem}

For $1\leq k\leq R$, we define mutually i.i.d. (extended) random variables $\xi_1,\cdots,\xi_t$ and $\bar \xi_1,\cdots,\bar \xi_t$ with cdf
\begin{align*}
    H(\eta):=1-\min \left\{ 1, (5/4)m_k^{-2/5}+(5/4)\Psi\left ( \frac{\eta+ \lambda_k n_k}{(n_k(1- \lambda_k^2))^{1/2}}\right )\right\},
\end{align*}
and
\begin{align*}
    \bar H(\eta):=\min \left\{ 1, (5/4)m_k^{-2/5}+(5/4)\Psi\left ( \frac{-\eta-\bar \lambda_k n_k}{(n_k(1-\bar \lambda_k^2))^{1/2}}\right )\right\}.
\end{align*}

\begin{lem}\label{l:kstep}
With probability larger than $1-\exp(-n^{1/72})$, for $n$ large enough, we have the sample cdf $\hat H$ of $\xi$ satisfies
\begin{align*}
    \bP_2(\sgn(S(0:k-1\mid \cR_k))S(k:k \mid \cR_k)\geq \eta)\leq 1-\hat H(\eta),
\end{align*}
and the sample cdf $\hat H_2$ of $\bar \xi$ satisfies
\begin{align*}
    \bP_2(\sgn(S(0:k-1\mid \cR_k))S(k:k \mid \cR_k)\leq \eta)\leq \hat H_2(\eta).
\end{align*}
Further, we can define $\xi_1, \cdots, \xi_t$ on the same space such that with probability larger than $1-\exp(-n^{1/73})$, for $n$ large enough,
\begin{align*}
    \sgn(S^{(r_i)}(0:k-1))S^{(r_i)}(k:k)\leq \xi_i,
\end{align*}
and define $\bar\xi_1, \cdots, \bar \xi_t$ on the same space such that with probability larger than $1-\exp(-n^{1/73})$, for $n$ large enough,
\begin{align*}
    \sgn(S^{(r_i)}(0:k-1))S^{(r_i)}(k:k)\geq \bar \xi_i.
\end{align*}
\end{lem}
\begin{proof}
This follows by combining the previous lemmas.
\end{proof}

We now bound the cdf of $S(k:k\mid \cR_k^c)$. 
\begin{lem}\label{l:kstep2}
For $1\leq k<R$, let $t=m-m_k$ and $\{r_1,\cdots,r_t\}=\cR_k^c$. Then we can define mutually iid (extended) random variables $\xi_1,\cdots,\xi_t$ with cdf
\begin{align*}
    H(\eta):=1-\min \{1, (5/4)m_k^{-1/6}+(5/4)\Psi(\eta/n_k^{1/2})\},
\end{align*}
and $\bar \xi_1,\cdots,\bar \xi_t$ with cdf
\begin{align*}
    H(\eta):=\min \{1, (5/4)m_k^{-1/6}+(5/4)\Psi(-\eta/n_k^{1/2})\},
\end{align*}
on the same space as $\{S^{(r)}(k:k)\}_{r\in \cR_k^c}$ such that with probability larger than $1-\exp(-n^{1/73})$, for $n$ large enough,
\begin{align*}
    \bar\xi_r\leq  \sgn(S^{(r_i)}(0:k-1))S^{(r_i)}(k:k) \leq \xi_i,
\end{align*}
for any $r_i\in \cR_k^c$.
\end{lem}
\begin{proof}
This follows from Lemma \ref{l:binomial0}.
\end{proof}

To combine the arguments, we state the following lemma.
\begin{lem}[Lemma 5.16 in \cite{kim1998covering}]\label{l:combine}
Suppose $\xi,\xi',\psi,\psi'$ are random variables with cdfs $F_\xi$, $F_{\xi'}$, $F_\psi$, $F_{\psi'}$ respectively. If
\begin{align*}
    F_{\xi}\leq a+bF_{\xi'},
\end{align*}
and 
\begin{align*}
    F_{\psi}\leq a+bF_{\psi'},
\end{align*}
where $\xi$ and $\psi$ and $\xi'$ and $\psi'$ are pairwisely independent. Then we have
\begin{align*}
    F_{\xi+\psi}\leq \min\{(c+ad)+bdF_{\xi'+\psi'},(a+bc)+bdF_{\xi'+\psi'}\}.
\end{align*}
\end{lem}

\begin{lem}[Lemma 5.15 in \cite{kim1998covering}]\label{l:bigabs}
Let $m$ and $t$ be integers with $1\leq t\leq m$, let $T$ be some threshold. Let $\bar F(\cdot)$ be a nonnegative nondecreasing function from $\bR$ to $\bR$. Suppose that the collection $\{\hat\xi_1,\cdots,\hat\xi_m\}$ has empirical cdf $\hat F(\cdot)$, where $\hat F(u)\leq \bar F(u)$ for all $u\in \bR$ and $\bar F(T)\leq t/m$. Consider a permutation of $\hat \xi_i$'s with $\hat\xi^{(1)}\leq \cdots \leq \hat\xi^{(m)}$. Then for $t\in [m]$ such that $\bar F(T)\leq t/m$, the sample cdf $\hat H_t(\cdot)$ of $\{\hat \xi^{(1)},\cdots, \hat \xi^{(1)}\}$ satisfies
\begin{align*}
    \hat H_t(u)\leq \frac{m}{t} \bar F(u)
\end{align*}
for all $u \in \bR$.
\end{lem}

Combining all the above lemmas, we have the following lemma that controls the behaviour of $S(0:k)$.

\begin{lem}\label{l:together}
For any $\eta\geq T_k$, for all $k\in \{1,2,\cdots, R-1\}$, we have with probability larger than $1-k\exp(-n^{1/75})$, for $n$ large enough,
\begin{align}\label{eq:inductbound}
    \bP_2(|S(0:k)|\geq \eta)\leq 2\cdot 3^k\left( n^{-1/10} +\Psi \left( \frac{\eta+\mu_{k,1}}{\sigma_{k,1}}\right)+(3/2)(1+n^{-1/10})^k \Psi \left( \frac{\eta+\mu_{k,2}}{\sigma_{k,2}}\right)\right),
\end{align}
where 
\begin{align*}
    \mu_{k,1}=-T_{k-1}, \quad \sigma_{k,1}^2=n_k, \quad \mu_{k,2}=\sum_{i=1}^k n_i\sqrt{2/\pi m_i}, \quad \sigma_{k,2}^2=\sum_{i=0}^k n_i.
\end{align*}
\end{lem}

\subsection{Bounds on the last round in the SBP model}
In this section, we prove Lemma \ref{l:connect} provided all the lemmas in the previous section.
\begin{lem}\label{l:last1}
For each $r\in [m]$ and every $\veps\in (0,1/6)$, for $n$ large enough,
\begin{align*}
    \bP(|S^{(r)}(0:R-1)|\geq n^{1/2+\veps})\leq \exp(-n^{\veps}).
\end{align*}
\end{lem}
\begin{proof}
Note that by Lemma \ref{l:ulbound}, for each $r\in [m]$, $S^{(r)}(0:R-1)$ can be bounded from below and above by a sum of independent Bernoulli random variables,
\begin{align*}
    \sum_{k=0}^{R-1}\sum_{j=1}^{n_k}  B_{k,j}\leq S^{(r)}(0:R-1)\leq \sum_{k=0}^{R-1} \sum_{j=1}^{n_k} \bar B_{k,j},
\end{align*}
where $\bar B_{k,j}\sim \mathrm{Bern}((1+\bar \lambda_k)/2)$ and $B_{k,j}\sim \mathrm{Bern}((1-\bar \lambda_k)/2)$. Here we write $\bar \lambda_0=0$. As the mean
\begin{align*}
    \sum_{k=0}^{R-1} n_k\bar \lambda_k\leq (2C_\kappa+R\kappa)\sqrt{n} \ll n^{1/2+\veps},
\end{align*}
by Lemma \ref{l:binomial0}, we have the lemma.
\end{proof}

Next we look at the general row sums $S(v,h_0,h_1,\cdots,h_{R_1})$. 
\begin{lem}\label{l:last2}
For any $\vec 0\leq \vec h \leq \vec D$, and for each $r\in [m]$ and every $\veps\in (0,1/6)$, for $n$ large enough,
\begin{align*}
    \bP(|S^{(r)}(v,h_0,h_1,\cdots,h_{R-1})|\geq n^{1/2+\veps})\leq \exp(-n^{\veps}).
\end{align*}
\end{lem}
\begin{proof}
The proof is similar to Lemma \ref{l:last1}. From construction, we have for any $r\in [m]$,
\begin{align*}
    \sum_{k=0}^{R-1}\sum_{j=1}^{n_k}  B_{k,j}\leq S^{(r)}(v,h_0,h_1,\cdots,h_{R-1})\leq \sum_{k=0}^{R-1} \sum_{j=1}^{n_k} \bar B_{k,j},
\end{align*}
where $\bar B_{k,j}\sim \mathrm{Bern}((1+\bar \lambda_k)/2)$ and $B_{k,j}\sim \mathrm{Bern}((1-\bar \lambda_k)/2)$ and $\bar \lambda_0=0$. As the mean
\begin{align*}
    \sum_{k=0}^{R-1} n_k\bar \lambda_k\leq (2C_\kappa+R\kappa)\sqrt{n} \ll n^{1/2+\veps},
\end{align*}
by Lemma \ref{l:binomial0}, we have the lemma.
\end{proof}
Therefore, by a union bound, we have
\begin{align}\label{eq:2big}
    \bP(|S^{(r)}(v,h_0,h_1,\cdots,h_{R-1}))|\geq n^{0.501}, \forall r\in [m])\leq m\exp(-n^{0.001}).
\end{align}
We use $\cR_R$ to denote the set of rows involved in computation for $S^{(r)}(v,h_0,h_1,\cdots,h_{R-1})$ in round $R$, and $\cR_R'$ for $S^{(r)}(v,h_0,h_1,\cdots,h_{R-1})^{+1}$. Then for $r\in \cR_R$, by definition we can write
\begin{align*}
    S^{(r)}(h_0,h_1,\cdots,h_{R-1},0)(R:R)= &- \sgn(S^{(r)}(h_0,h_1,\cdots,h_{R-1})) \sum_{j=1}^{\ell_r} \cW_{r,j} \sgn(\sum_{r\in \cR_R} \cW_{r,j}) \\
    &+ \sum_{j=\ell_r+1}^{n_R} G_{r,j}X_j,
\end{align*}
and similarly for $r\in \cR_R'$,
\begin{align*}
    S^{(r)}(h_0,h_1,\cdots,h_{R-1},D_R)(R:R)= & - \sgn(S^{(r)}(h_0,h_1,\cdots,h_{R-1})^{+1}) \sum_{j=1}^{\ell'_r} \cW'_{r,j} \sgn(\sum_{r\in \cR'_R} \cW'_{r,j}) \\
    &+ \sum_{j=\ell'_r+1}^{n_R} G_{r,j}X_j.
\end{align*}
Note that from construction, $S^{(r)}(h_0,h_1,\cdots,h_{R-1})$ and $S^{(r)}(h_0,h_1,\cdots,h_{R-1})^{+1}$ differ by $2$. If they have the same sign, we use $\sgn(S^{(r)}(0:R-1))$ to denote it. Then for any $r\in \cR_R\cap \cR_R'$, we have for any $0\leq h_R\leq D_R$,
\begin{align*}
    S^{(r)}(h_0,h_1,\cdots,h_{R-1},h_R)(R:R)= &- \sgn(S^{(r)}(0:R-1)) \sum_{j=1}^{\ell''_r} \cW''_{r,j} \sgn(\sum_{r\in \cR''_R} \cW''_{r,j}) \\
    &+ \sum_{j=\ell''_r+1}^{n_R} G_{r,j}X_j,
\end{align*}
where $\ell_r''$ is an integer between $\ell_r$ and $\ell'_r$, and each term $\cW''_{r,j} \sgn(\sum_{r\in \cR''_R} \cW''_{r,j})$ either equals $\cW_{r,j} \sgn(\sum_{r\in \cR_R} \cW_{r,j})$ or $\cW'_{r,j} \sgn(\sum_{r\in \cR'_R} \cW'_{r,j})$. We write $\vec{h}=(h_0,\cdots,h_R)$ and write $S^{(r)}(\vec h)=S^{(r)}(v,h_0,h_1,\cdots,h_{R-1},h_R)$ for short. We write $\bP_0$ as the probability with respect to the randomness in $G(R:R)$ only. By Lemma \ref{l:binomial0}, 
\begin{align*}
    &\bP_0(|S^{(r)}(\vec h)(0:R-1)+S^{(r)}(\vec h)(R:R)|\geq \kappa \sqrt{n})\\
    \leq &\bP_0(S^{(r)}(\vec h)(0:R-1)+S^{(r)}(\vec h)(R:R)\geq \kappa \sqrt{n})+\bP(S^{(r)}(\vec h)(0:R-1)+S^{(r)}(\vec h)(R:R)\leq -\kappa \sqrt{n})\\
    \leq &2\exp(-m^{1/6})+\frac{7}{3}\Psi\left( \frac{\kappa \sqrt{n}-|S^{(r)}(\vec h)(0:R-1)|+ \ell''_r \lambda_R}{ \sqrt{n_R} }\right)\\
    &\qquad\qquad\qquad+\frac{7}{3}\Psi\left( \frac{\kappa \sqrt{n}+|S^{(r)}(\vec h)(0:R-1)|- \ell''_r \bar \lambda_R}{ \sqrt{n_R} }\right).
\end{align*}
Note that by construction, 
\begin{align*}
    \ell_r\leq |S^{(r)}(h_0,\cdots,h_{R-1})(0:R-1)|\sqrt{\pi m_R /2} \leq \ell_r+1,\\
    \ell'_r\leq |S^{(r)}(h_0,\cdots,h_{R-1})^{+1}(0:R-1)|\sqrt{\pi m_R /2} \leq \ell'_r+1.
\end{align*}
This implies that $\ell_r$ and $\ell_r'$ can only differ by at most one. In the event that $|S^{(r)}(\vec h)(0:R-1)|< n^{0.501}, \forall r\in [m]$, we have $n_R\gg \ell_r''$. Therefore, together with equation \eqref{eq:2big}, we have that
\begin{align*}
    \bP(|S^{(r)}(\vec h)(0:R-1)+S^{(r)}(\vec h)(R:R)|\geq \kappa \sqrt{n}) \leq \exp(-n^{0.0005}).
\end{align*}
Now, if row $r\in \cR_R\backslash \cR_r'$ or $r\in \cR_R'\backslash \cR_r$, then the row sum in round $R$ consists partly of majority votes and the rest of Bernoulli random variables. Similarly, we can have that
\begin{align*}
    &\bP_0(|S^{(r)}(\vec h)(0:R-1)+S^{(r)}(\vec h)(R:R)|\geq \kappa \sqrt{n})\\
    \leq &2\exp(-m^{1/6})+\frac{7}{3}\Psi\left( \frac{\kappa \sqrt{n}-|S^{(r)}(\vec h)(0:R-1)|+ \ell''_r \lambda_R}{ \sqrt{n_R} }\right)\\
    &\qquad\qquad\qquad+\frac{7}{3}\Psi\left( \frac{\kappa \sqrt{n}+|S^{(r)}(\vec h)(0:R-1)|- \ell''_r \bar \lambda_R}{ \sqrt{n_R} }\right),
\end{align*}
where $\ell_r''$ here is the number of majority votes. Note that in particular in this case, by Lemma \ref{l:together2},
\begin{align*}
    |S^{(r)}(\vec h)(0:R-1)|\leq T_{R-1}+2.
\end{align*}
As $n_R\ll \kappa \sqrt{n}-T_{R-1}$, we also have that
\begin{align*}
    \bP(|S^{(r)}(\vec h)(0:R-1)+S^{(r)}(\vec h)(R:R)|\geq \kappa \sqrt{n}) \leq \exp(-n^{0.0005}).
\end{align*}
For row $r\in (\cR_R\cup\cR_R')$ where $S^{(r)}(h_0,h_1,\cdots,h_{R-1})$ and $S^{(r)}(h_0,h_1,\cdots,h_{R-1})^{+1}$ have a different sign, it is easy to see that $|S^{(r)}(h_0,h_1,\cdots,h_{R-1})|=1$. Then the argument is much simpler and we directly have that 
\begin{align*}
    \bP(|S^{(r)}(\vec h)(0:R-1)+S^{(r)}(\vec h)(R:R)|\geq \kappa \sqrt{n}) \leq \exp(-n^{0.0005}).
\end{align*}
For rows $r\in (\cR_R\cup \cR_R')^c$, the row sum in round $R$ only consists of sums of Bernoulli random variables. Therefore, by Lemma \ref{l:together2}, as $n_R\ll \sqrt{n}-T_{R-1}$, we also have that 
\begin{align*}
    \bP(|S^{(r)}(\vec h)(0:R-1)+S^{(r)}(\vec h)(R:R)|\geq \kappa \sqrt{n}) \leq \exp(-n^{0.0005}).
\end{align*}
By a union bound over all rows $r\in [m]$, this completes the proof of Lemma \ref{l:connect}, and thus proves Theorem \ref{t:octupus} and Theorem \ref{t:algo} in the SBP model.  

\subsection{Inductive bound on distribution in ABP model}
In this section, we consider the ABP model. We adapt the same notations as in the previous sections. The following lemma will be important in proving Lemma \ref{l:connect}.
\begin{lem}\label{l:together3}
For any $0\leq k \leq R-1$, $\eta\leq T_k$, and for any $\vec 0\leq \vec h \leq \vec D$,
\begin{align}\label{eq:inductbound2}
    \bP_2[S(h_0,h_1,\cdots,h_k)\leq \eta] \leq 3^k\left( n^{-1/10} +\Psi \left( \frac{-\eta+\mu_{k,1}}{\sigma_{k,1}}\right)+(3/2)(1+n^{-1/10})^k \Psi \left( \frac{-\eta+\mu_{k,2}}{\sigma_{k,2}}\right)\right),
\end{align}
with probability larger than $1-\exp(-n^{1/80})$. Here, 
\begin{align*}
    \mu_{k,1}=T_{k-1}, \quad \sigma_{k,1}^2=n_k, \quad \mu_{k,2}=\sum_{i=1}^k n_i\sqrt{2/\pi m_i}, \quad \sigma_{k,2}^2=\sum_{i=0}^k n_i.
\end{align*}
Furthermore, the right hand side is bounded from above by $m_{k+1}/m$ for any $1\leq k\leq R-1$.
\end{lem}

With Lemma \ref{l:together3}, we prove Lemma \ref{l:connect} in the rest of the subsection.
\begin{lem}
For each $r\in [m]$ and every $\veps\in (0,1/6)$, for $n$ large enough,
\begin{align*}
    \bP(S^{(r)}(0:R-1)\leq -n^{1/2+\veps})\leq \exp(-n^{\veps}).
\end{align*}
\end{lem}
\begin{proof}
Note that $S^{(r)}(0:R-1)$ can be bounded from below by a sum of iid Bernoulli random variables that follow the distribution $\mathrm{Bern}(1/2)$.  Therefore, the lemma holds by Lemma \ref{l:binomial0}.
\end{proof}
Next we look at the general row sums $S(v,h_0,h_1,\cdots,h_{R-1})$. 
\begin{lem}
For any $\vec 0\leq \vec h \leq \vec D$, and for each $r\in [m]$ and every $\veps\in (0,1/6)$, for $n$ large enough,
\begin{align*}
    \bP(S^{(r)}(v,h_0,h_1,\cdots,h_{R-1})(0:R-1)\leq -n^{1/2+\veps})\leq \exp(-n^{\veps}).
\end{align*}
\end{lem}
\begin{proof}
The argument is very similar to the previous subsection. By Lemma \ref{l:ulbound}, we have for any $r\in [m]$,
\begin{align*}
    S^{(r)}(v,h_0,h_1,\cdots,h_{R-1})\geq \sum_{k=0}^{R-1} \sum_{j=1}^{n_k} B_{k,j},
\end{align*}
where for each $k$, $B_{k,1}\sim \mathrm{Bern}((1-\bar \lambda_k)/2)$ and for $j>1$, $B_{k,j}\sim \mathrm{Bern}(1/2)$ and $\bar \lambda_0=0$. As the mean
\begin{align*}
    \sum_{k=0}^{R-1} -\bar \lambda_k\gg -n^{1/2+\veps},
\end{align*}
by Lemma \ref{l:binomial0}, we have the lemma.
\end{proof}
Therefore, by a union bound, we have
\begin{align}\label{eq:2big2}
    \bP(S^{(r)}(0:R-1))\leq - n^{0.501}, \forall r\in [m])\leq m\exp(-n^{0.001}).
\end{align}
We adopt a similar set of definitions as in the SBP model. We use $\cR_R$ to denote the set of rows involved in computation for $S^{(r)}(v,h_0,h_1,\cdots,h_{R-1})$ in round $R$, and $\cR_R'$ for $S^{(r)}(v,h_0,h_1,\cdots,h_{R-1})^{+1}$.
By definition we can write
\begin{align*}
    &S^{(r)}(h_0,h_1,\cdots,h_{R-1},0)(R:R)=\sum_{j=n-n_R+1}^{n} G_{r,j} \sgn(\sum_{t\in \cR_R} G_{t,j}), \quad r\in \cR_R\\
    &S^{(r)}(h_0,h_1,\cdots,h_{R-1},D_R)(R:R)=\sum_{j=n-n_R+1}^{n} G_{r,j} \sgn(\sum_{t\in \cR'_R} G_{t,j}), \quad r\in \cR_R'.
\end{align*}
Then for any $0\leq h_R\leq D_R$, and any $r\in \cR_R \cap \cR_R'$, we have that 
\begin{align*}
    &S^{(r)}(h_0,h_1,\cdots,h_{R-1},h_R)(R:R)=\sum_{j=n-n_R+1}^{L(h_R)} G_{r,j} \sgn(\sum_{t\in \cR_R} G_{t,j})+\sum_{j=L(h_R)+1}^{n} G_{r,j} \sgn(\sum_{t\in \cR'_R} G_{t,j}),
\end{align*}
where $L(h_R)$ is between $n-n_R+1$ and $n$ such that $T(h_0,\cdots,h_R)$ agrees with $T(h_0,\cdots,D_R)$ in the first $L(h_R)$ entries. We write $\vec{h}=(h_0,\cdots,h_R)$ and write $S^{(r)}(\vec h)=S^{(r)}(h_0,h_1,\cdots,h_{R-1},h_R)$ for short. Further, we write $\bP_0$ as the probability with respect to the randomness in $G(R:R)$ only. Then we have,
\begin{align*}
    &\bP_0(S^{(r)}(\vec h)(0:R-1)+S^{(r)}(\vec h)(R:R)\leq \kappa \sqrt{n})\\
    \leq &\exp(-m^{1/6})+\frac{7}{6}\Psi\left( \frac{-\kappa \sqrt{n}+S^{(r)}(\vec h)(0:R-1)+ n_R \lambda_R}{ \sqrt{n_R} }\right).
\end{align*}
In the event that $S^{(r)}(\vec h)(0:R-1)>- n^{0.501}, \forall r\in [m]$, we have $n_R\lambda_R\gg n^{0.501}$. Therefore, together with equation \eqref{eq:2big2}, we have that
\begin{align*}
    \bP(S^{(r)}(\vec h)(0:R-1)+S^{(r)}(\vec h)(R:R)\leq \kappa \sqrt{n}) \leq \exp(-n^{0.0005}).
\end{align*}
Note that $S^{(r)}(h_0,\cdots,h_{R-1})$ and $S^{(r)}(h_0,\cdots,h_{R-1})^{+1}$ can only differ by two. By Lemma \ref{l:together3}, for row $r$ such that $r\in (\cR_R\cap \cR'_R)^c$, we have
\begin{align*}
    &\bP_0(S^{(r)}(\vec h)(0:R-1)+S^{(r)}(\vec h)(R:R)\leq \kappa \sqrt{n})\\
    \leq &\exp(-m^{1/6})+\frac{7}{6}\Psi\left( \frac{-\kappa \sqrt{n}+T_{R-1}-2}{ \sqrt{n_R} }\right).
\end{align*}
Note that we have $-\kappa \sqrt{n}+T_{R-1}\gg \sqrt{n_R}$. Therefore, we have that
\begin{align*}
    \bP(S^{(r)}(\vec h)(0:R-1)+S^{(r)}(\vec h)(R:R)\leq \kappa \sqrt{n}) \leq \exp(-n^{0.005}).
\end{align*}
By a union bound over all rows $r\in [m]$, this completes the proof of Lemma \ref{l:connect}, and thus proves Theorem \ref{t:octupus} and Theorem \ref{t:algo} in the ABP model.  

\subsection{Proofs of Lemmas}
\subsubsection{Proof of Lemma \ref{l:round1sum}}
This following lemma controls the behaviour of binomial distributions and mixed binomial distributions.
\begin{lem}\label{l:binomial}
For any $\varepsilon>0$. For any $p_1,p_2\in [1/10,9/10]$, and $m$ large enough, we have
\begin{align*}
    &\bP(B_1+B_2\leq \eta)\leq \exp(-m^{1/6})+(1+\varepsilon)\Psi\left(\frac{-\eta+m_1p_1+m_2p_2}{(m_1p_1(1-p_1)+m_2p_2(1-p_2))^{1/2}}\right),\\
    &\bP(B_1+B_2\geq \eta)\leq \exp(-m^{1/6})+(1+\varepsilon)\Psi\left(\frac{\eta-m_1p_1-m_2p_2}{(m_1p_1(1-p_1)+m_2p_2(1-p_2))^{1/2}}\right),
\end{align*}
where $B_1$ and $B_2$ are independent with $B_1\sim Bin(m_1,p_1)$, $B_2\sim Bin(m_2,p_2)$, $m_1+m_2=m$.
\end{lem}
\begin{proof}
The proof is very similar to the proof of Lemma \ref{l:binomial0}.
\end{proof}
Next we prove Lemma \ref{l:round1sum}.
\begin{proof}[Proof of Lemma \ref{l:round1sum}]
Note that for each $r$,
\begin{align*}
    S^{(r)}(0:1)&= S^{(r)}(0:0)+S^{(r)}(1:1)\\
    &= S^{(r)}(0:0) - \sgn(S^{(r)}(0:0)) \sum_{j=1}^{\ell_r} W_{r,j} \sgn(\sum_{r\in [m]} W_{r,j}) + \sum_{j=\ell_r+1}^{n_1} G_{r,j}X_j.
\end{align*}
For any $\cB$ $b$-element subset of $\{1,\cdots, m\}$, for any $r\in \cB$, by Lemma \ref{l:ulbound}, we can define i.i.d. $\psi_{j,r}\sim \mathrm{Bern}(\frac{1}{2}(1+(1-m^{-1/8})\sqrt{2/(\pi m )}))$ and i.i.d. $\bar \psi_{j,r} \sim  \mathrm{Bern}(\frac{1}{2}(1+(1+m^{-1/8})\sqrt{2/(\pi m )}))$ such that $\psi_{j,r}\leq W_{r,j} \sgn(\sum_{r\in [m]} W_{r,j}) \leq \bar \psi_{j,r}$ and thus
\begin{align}\label{eq:srbound10}
    \sgn(S^{(r)}(0:0))S^{(r)}(0:1)\leq |S^{(r)}(0:0)| -  \sum_{j=1}^{\ell_r}  \psi_{j,r} + \sum_{j=\ell_r+1}^{n_1}  Q_{r,j},
\end{align}
where $Q_{r,j}$ are i.i.d. Rademacher random variables. And 
\begin{align}\label{eq:srbound20}
    \sgn(S^{(r)}(0:0))S^{(r)}(0:1)\geq |S^{(r)}(0:0)| - \sum_{j=1}^{\ell_r} \bar \psi_{j,r} + \sum_{j=\ell_r+1}^{n_1} \bar Q_{r,j},
\end{align}
where $\bar Q_{r,j}$ are i.i.d. Rademacher random variables. Note here if $\ell_r=n_1$, then we assume the third term in zero. We use $\bP_0$ to denote the probability density with respect to the randomness of $G(1:1)$ only. Then for fixed $S^{(r)}(0:0)$, when $\ell_r\leq n_1$, by Lemma \ref{l:binomial}, we have for any $\eta>0$,
\begin{align*}
    \bP_0(-  \sum_{j=1}^{\ell_r}  \psi_{j,r} + \sum_{j=\ell_r+1}^{n_1}  Q_{r,j}\geq \eta)\leq \exp(-m^{1/6})+\frac{21}{20}\Psi\left( \frac{\eta+ \ell_r \lambda_1}{ \sqrt{\ell_r(1-\lambda_k^2)+(n_1-\ell_r)} }\right).
\end{align*}
As $\ell_r=\lfloor |S^{(r)}(0:0)|\sqrt{\pi m /2} \rfloor$, this implies that for any $\eta>0$,
\begin{align*}
    \bP_0(|S^{(r)}(0:0)|-  \sum_{j=1}^{\ell_r}  \psi_{j,r} &+ \sum_{j=\ell_r+1}^{n_1}  Q_{r,j}\geq \eta)
    \leq \exp(-m^{1/6})+\frac{21}{20}\Psi\left( \frac{\eta-|S^{(r)}(0:0)|+ \ell_r \lambda_1}{ \sqrt{\ell_r(1-\lambda_k^2)+(n_1-\ell_r)} }\right)\\
    &\leq \exp(-m^{1/6})+\frac{21}{20}\Psi\left( \frac{\eta-|S^{(r)}(0:0)|m^{-1/8}-\lambda_1}{ \sqrt{\ell_r(1-\lambda_k^2)+(n_1-\ell_r)} }\right)\\
    &\leq \exp(-m^{1/6})+\frac{21}{20}\Psi\left( \frac{\eta-\sqrt{2/\pi m}((1+n_1)m^{-1/8}+(1-m^{-1/8}))}{ \sqrt{n_1} }\right).\\
\end{align*}
And for the other direction, for any $\eta<0$, we have
\begin{align*}
    \bP_0(|S^{(r)}(0:0)|-  \sum_{j=1}^{\ell_r}  \bar\psi_{j,r} + \sum_{j=\ell_r+1}^{n_1}  \bar Q_{r,j}\leq \eta)
    &\leq \exp(-m^{1/6})+\frac{21}{20}\Psi\left( \frac{-\eta+|S^{(r)}(0:0)|- \ell_r \bar \lambda_1}{ \sqrt{\ell_r(1-\bar \lambda_k^2)+(n_1-\ell_r)} }\right)\\
    &\leq \exp(-m^{1/6})+\frac{21}{20}\Psi\left( \frac{-\eta-|S^{(r)}(0:0)|m^{-1/8}}{ \sqrt{\ell_r(1-\lambda_k^2)+(n_1-\ell_r)} }\right)\\
    &\leq \exp(-m^{1/6})+\frac{21}{20}\Psi\left( \frac{-\eta-\sqrt{2/\pi m}(1+n_1)m^{-1/8}}{ \sqrt{n_1} }\right).
\end{align*}
So when $\ell_r\leq n_1$, we have for any $\eta>0$, 
\begin{align*}
    \bP_0(|S^{(r)}(0:1)|\geq \eta)\leq 2\exp(-m^{1/6})+\frac{21}{10}\Psi\left( \frac{\eta-\sqrt{2/\pi m}((1+n_1)m^{-1/8}+(1-m^{-1/8}))}{ \sqrt{n_1} }\right).
\end{align*}
Now, if $\ell_r>n_1$, we have for any $\eta>0$,
\begin{align}
    \bP(- \sum_{j=1}^{n_1}  \psi_{j,r}\geq \eta)
    &\leq \exp(-m^{1/6})+\frac{21}{20}\Psi\left( \frac{\eta+ n_1 \lambda_1}{ \sqrt{\ell_r(1-\lambda_k^2)+(n_1-\ell_r)} }\right)\\
    &\leq \exp(-m^{1/6})+\frac{21}{20}\Psi\left( \frac{\eta+ n_1 (1-m^{-1/8})\sqrt{2/(\pi m)}}{\sqrt{n_1}  }\right).\label{eq:psibound}
\end{align}
And when $\eta<0$, we have
\begin{align*}
    \bP_0(|S^{(r)}(0:0)|-  \sum_{j=1}^{n_1}  \bar \psi_{j,r}\leq \eta)
    &\leq \exp(-m^{1/6})+\frac{21}{20}\Psi\left( \frac{-\eta+|S^{(r)}(0:0)|- n_1 \bar \lambda_1}{ \sqrt{\ell_r(1-\lambda_k^2)+(n_1-\ell_r)} }\right)\\
    &\leq \exp(-m^{1/6})+\frac{21}{20}\Psi\left( \frac{-\eta - n_1 m^{-1/8}\sqrt{2/(\pi m)}}{\sqrt{n_1}  }\right).
\end{align*}
Therefore, note that as $S^{(r)}(0:0)$ is a binomial random variable, we have for any $\eta>0$,
\begin{align*}
    \bP(|S^{(r)}(0:1)|\geq \eta)&\leq \bP(\ell_r\leq n_1) \left( 2\exp(-m^{1/6})+\frac{21}{10}\Psi\left( \frac{\eta-\sqrt{2/\pi m}(1+n_1)m^{-1/8}}{ \sqrt{n_1} }\right) \right)\\
    &+\bP(\ell_r> n_1) \left( \exp(-m^{1/6})+\frac{21}{20}\Psi\left( \frac{\eta + n_1 m^{-1/8}\sqrt{2/(\pi m)}}{\sqrt{n_1}  }\right) \right)\\
    &+\bP(\ell_r> n_1) \bP(\sgn(S^{(r)}(0:0))S^{(r)}(0:1)\geq \eta \mid \ell_r> n_1).
\end{align*}
Note that condition on $\ell_r> n_1$,
\begin{align}
    \bP(|S^{(r)}(0:0)|\geq \eta \mid \ell_r> n_1)\leq  \left( 2\exp(-m^{1/6})+\frac{21}{20}\Psi\left( \frac{\eta}{\sqrt{n_0}  }\right) \right)/\bP(\ell_r> n_1).\label{eq:sbound}
\end{align}
Therefore, by combining equation \eqref{eq:psibound} and \eqref{eq:sbound} and Lemma \ref{l:combine}, we have
\begin{align*}
    &\bP(\ell_r> n_1) \bP(\sgn(S^{(r)}(0:0))S^{(r)}(0:1)\geq \eta \mid \ell_r> n_1)\\
    &\leq 4\exp(-m^{-1/6})+\frac{441}{200}\Psi\left( \frac{\eta+n_1(1-m^{-1/8})\sqrt{2/(\pi m)}}{\sqrt{n_0+n_1}  }\right).
\end{align*}
Therefore, all together, for $m$ large enough, we have
\begin{align*}
    \bP(|S^{(r)}(0:1)|\geq \eta)\leq \exp(-m^{1/5})+\frac{7}{3}\Psi\left( \frac{\eta}{\sqrt{n_1}  }\right)+\frac{7}{3}\Psi\left( \frac{\eta+n_1\sqrt{2/(\pi m)}}{\sqrt{n_0+n_1}  }\right).
\end{align*}
This holds for any $r \in \cB$. We deal with the depedence of rows in the following and thus prove the lemma. Note that this is rather similar to the proof of Lemma 5.5 in \cite{kim1998covering}, therefore we only record a few key arguments. We firstly note that for any $T\in [b,t-1]$,
\begin{align*}
    \bP\left( \sum_{i=1}^t \xi_i\geq T \right) \leq \frac{\binom{t}{b} q^b }{\binom{T}{b}}.
\end{align*}
To prove this, we notice that
\begin{align*}
    \E \left( \sum_{B \subseteq [t], |B|=b} I(\xi_i=1, \forall i \in B) \right) \leq \binom{t}{b} q^b.
\end{align*}
When $\sum_{i=1}^t \xi_i \geq T \geq b$, we have
\begin{align*}
    \sum_{B \subseteq [t], |B|=b} I(\xi_i=1, \forall i \in B)\geq \binom{T}{b}.
\end{align*}
Therefore the results follows by Markov's inequality. From here, the arguments are the same as in the proof of Lemma 5.5 in \cite{kim1998covering} and we omit the rest.
\end{proof}

\subsubsection{Proof of Lemma \ref{l:together}}
We firstly present the following two lemmas.
\begin{lem}\label{l:induct2}
The right hand side of \eqref{eq:inductbound0} (take $\eta=T_1$) is bounded by $m_2/m$. More precisely,
\begin{align*}
    (9/4)n_1^{-2/5}+(5/2)\Psi\left ( \frac{T_1+\mu_{1,1}}{\sigma_{1,1}}\right )+(5/2)\Psi\left ( \frac{T_1+\mu_{1,2}}{\sigma_{1,2}}\right )\leq m_{2}/m.
\end{align*}
\end{lem}
\begin{proof}
We check that the left hand side is bounded by 
\begin{align*}
    (5/2)\Psi\left ( \frac{9\kappa /20}{\sqrt{C_\kappa}(\alpha\pi /2)^{1/4}}\right )+(5/2)\Psi\left (C_\kappa\right )+o_n(1)
    \leq  \Psi\left ( \frac{\kappa}{2}+\frac{5}{\kappa}+4\right ),
\end{align*}
when $n $ is big.
\end{proof}

\begin{lem}\label{l:induct3}
For any $1< k\leq R$, the right hand side of \eqref{eq:inductbound} (take $\eta=T_k$) is bounded by $m_{k+1}/m$ for large enough $m$. More precisely,
\begin{align*}
    2\cdot 3^k\left( n^{-1/10} +\Psi \left( \frac{T_k+\mu_{k,1}}{\sigma_{k,1}}\right)+(3/2)(1+n^{-1/10})^k \Psi \left( \frac{T_k+\mu_{k,2}}{\sigma_{k,2}}\right)\right)\leq m_{k+1}/m.
\end{align*}
\end{lem}
\begin{proof}
The left hand side is bounded from above by
\begin{align*}
    &2\cdot 3^k \Psi \left( \frac{\frac{9}{10} \kappa/2^k}{ \sqrt{\kappa}\Psi( k\kappa/4+5+5/\kappa)^{1/4}}\right)+4 \cdot 3^k \Psi \left( k\kappa/2+5+5/\kappa\right)+o_n(1).
\end{align*}
Note that as 
\begin{align*}
\frac{\frac{9}{10} \kappa/2^k}{ \sqrt{\kappa}\Psi( k\kappa/4+5+5/\kappa)^{1/4}}> k\kappa/2+5+5/\kappa,
\end{align*}
and
\begin{align*}
    3^k \Psi \left( k\kappa/2+5+5/\kappa\right)\leq \Psi \left( k\kappa/4+5+5/\kappa\right),
\end{align*}
we have that the inequality holds.
\end{proof}

\begin{proof}[Proof of Lemma \ref{l:together}]
This is shown inductively. Notice that
\begin{align*}
    \sgn(S(0:k-1))S(0:k)&=\sgn(S(0:k-1))(S(0:k-1)+S(k:k))\\
    &=|S(0:k-1)|+\sgn(S(0:k-1))S(k:k).
\end{align*}
We note that $S(0:k-1)$ is symmetric around $0$. Therefore, $|S(0:k-1)|$ and $\sgn(S(0:k-1))$ are independent. By Lemma \ref{l:step1p2}, we have that for any $\eta\geq 0$,
\begin{align*}
    \bP_2(|S(0:1)|\geq \eta)\leq (9/4)n_1^{-2/5}+(5/2)\Psi\left ( \frac{\eta+\mu_{1,1}}{\sigma_{1,1}}\right )+(5/2)\Psi\left ( \frac{\eta+\mu_{1,2}}{\sigma_{1,2}}\right ).
\end{align*}
So the induction hypothesis holds for $k=1$. Now for any $1<k<R$, by the induction hypothesis and Lemma \ref{l:bigabs}, we have that with probability larger than $1-(k-1)\exp(-n^{1/75})-\exp(-n^{1/73})$, the sample cdf of $|S^{(r)}(0:k-1\mid \cR_{k+1})|$ is bounded from below by
\begin{align*}
    1-\frac{2m 3^{k-1}}{m_{k}}\left( n^{-1/10} +\Psi \left( \frac{\eta+\mu_{k-1,1}}{\sigma_{k-1,1}}\right)+(3/2)(1+n^{-1/10})^{k-1} \Psi \left( \frac{\eta+\mu_{k-1,2}}{\sigma_{k-1,2}}\right)\right),
\end{align*}
for any $\eta\geq T_{k-1}$. Write $t=m_k$ and assume $\{r_1,\cdots,r_t\}$ are the $t$ elements of $\cR_{k}$. By Lemma \ref{l:sam2stat} and \ref{l:sam2stat2}, we can define iid random variables $W_1,\cdots,W_t$ on the same space as $\{S^{(r)}(0:k-1)\}_{r\in \cR_{k}}$ with statistical cdf bounded from below by
\begin{align*}
    1-\frac{2m 3^{k-1}}{m_{k}}\left(  n^{-1/10} +\Psi \left( \frac{\eta+\mu_{k-1,1}}{\sigma_{k-1,1}}\right)+(3/2)(1+n^{-1/10})^{k-1} \Psi \left( \frac{\eta+\mu_{k-1,2}}{\sigma_{k-1,2}}\right)\right)-m_{k}^{-1/6},
\end{align*}
and that 
\begin{align*}
    W_i\geq |S^{(r_i)}(0:k-1)|, \quad \forall i \in [t],
\end{align*}
with probability larger than $1-(k-1)\exp(-n^{1/75})-2\exp(-n^{1/73})$. Condition on $\cR_{k}$, the random variables $|S^{(r)}(0:k-1)|$ and $\sgn(S^{(r)}(0:k-1))S(k:k)$ are independent. By Lemma \ref{l:kstep}, we can define iid random variables $\xi_1,\cdots,\xi_t$ and $\bar\xi_1,\cdots,\bar\xi_t$ which are also independent of $|S^{(r)}(0:k-1)|$ on the same space as $|S^{(r)}(0:k-1)|$ with cdf
\begin{align*}
    &1-\min \left\{ 1, (5/4)m_{k}^{-2/5}+(5/4)\Psi\left ( \frac{\eta+ \lambda_{k} n_{k}}{(n_{k}(1- \lambda_{k}^2))^{1/2}}\right )\right\},\\
    &\min \left\{ 1, (5/4)m_{k}^{-2/5}+(5/4)\Psi\left ( \frac{-\eta-\bar \lambda_{k} n_{k+1}}{(n_{k}(1-\bar \lambda_{k}^2))^{1/2}}\right)\right\},
\end{align*}
such that with probability larger than $1-(k-1)\exp(-n^{1/75})-3\exp(-n^{1/73})$, we have 
\begin{align*}
    \bar \xi_i \leq \sgn(S^{(r_i)}(0:k-1))S^{(r_i)}(k:k)\leq \xi_i.
\end{align*}
This implies that
\begin{align*}
    \bar \xi_i \leq \sgn(S^{(r_i)}(0:k-1))S(0:k)\leq W_i+\xi_i.
\end{align*}
By Lemma \ref{l:combine}, we have that for any $\eta>T_{k}$ and $n$ large enough,
\begin{align*}
    \bP(W_i+\xi_i\geq \eta)
    &\leq \frac{2m3^{k-1}}{m_{k}}\left(  2 n^{-1/10} +(5/4)\Psi \left( \frac{\eta+\mu_{k-1,1}+\lambda_{k}n_{k}}{(\sigma_{k-1,1}^2+n_{k}(1- \lambda_{k}^2))^{1/2}}\right)\right.\\
    &\left.+(15/8)(1+n^{-1/10})^{k-1} \Psi \left( \frac{\eta+\mu_{k-1,2}+\lambda_{k}n_{k}}{(\sigma_{k-1,2}^2+n_{k}(1- \lambda_{k}^2))^{1/2}}\right)\right).
\end{align*}
And for any $\eta<-T_{k}$ and $n$ large enough,
\begin{align*}
    \bP(\bar \xi_i\leq \eta) 
    \leq (5/4)m_{k}^{-2/5}+(5/4)\Psi\left ( \frac{-\eta-\bar \lambda_{k} n_{k}}{(n_{k}(1-\bar \lambda_{k}^2))^{1/2}}\right).
\end{align*}
Further, by Lemma \ref{l:stat2sam}, we have a bound on the empirical cdf. With probability larger than $1-(k-1)\exp(-n^{1/75})-4\exp(-n^{1/73})$, we have
\begin{align*}
    &\bP_2(\sgn(S^{(r_i)}(0:k-1))S(0:k)\geq \eta)\leq \bP_2(W_i+\xi_i\geq \eta) \\
    &\leq t^{-2/5}+\frac{2m3^{k-1}(1+t^{-1/12})}{m_{k}}\left(  2 n^{-1/10} +(5/4)\Psi \left( \frac{\eta+\mu_{k-1,1}+\lambda_{k}n_{k}}{(\sigma_{k-1,1}^2+n_{k}(1- \lambda_{k}^2))^{1/2}}\right)\right.\\
    &\left.+(15/8)(1+n^{-1/10})^{k-1} \Psi \left( \frac{\eta+\mu_{k-1,2}+\lambda_{k}n_{k}}{(\sigma_{k-1,2}^2+n_{k}(1- \lambda_{k}^2))^{1/2}}\right)\right).
\end{align*}
And similarly, 
\begin{align*}
    &\bP_2(\sgn(S^{(r_i)}(0:k-1))S(0:k)\leq \eta)\leq \bP_2(\bar \xi_i\leq \eta)\\
    &\leq t^{-2/5}+(1+t^{-1/12})\left((5/4)m_{k}^{-2/5}+(5/4)\Psi\left ( \frac{-\eta-\bar \lambda_{k} n_{k}}{(n_{k}(1-\bar \lambda_{k}^2))^{1/2}}\right)\right).
\end{align*}
Furthermore, for rows $r\in \cR_{k}^c$, by Lemma \ref{l:induct2} and Lemma \ref{l:induct3}, we have that 
\begin{align*}
    |S^{(r)}(0:k-1)|\leq T_{k-1}, \quad \forall r\in \cR_{k},
\end{align*}
with probability larger than $1-(k-1)\exp(-n^{1/75})$. As $|S^{(r)}(0:k-1)|$ and $\sgn(S^{(r)}(0:k-1))S^{(r)}(k:k)$ are independent and the latter is the sum of iid Randemacher random variables, by Lemma \ref{l:kstep2}, for any $\eta>T_k$, we have
\begin{align*}
    \bP_2(|S^{(r)}(0:k\mid \cR_{k}^c)|\geq \eta)\leq (5/2)m_{k}^{-2/5}+(5/2)\Psi\left ( \frac{\eta-T_{k-1}}{(n_{k})^{1/2}}\right ),
\end{align*}
with probability larger than $1-(k-1)\exp(-n^{1/75})-6\exp(-n^{1/73})$. Combining the above, for any $\eta>T_k$ and $n$ large enough, we have
\begin{align*}
    &\bP_2(|S^{(r)}(0:k)|\geq \eta)\leq (5/2)m_{k}^{-2/5}+(5/2)\Psi\left ( \frac{\eta-T_{k-1}}{(n_{k})^{1/2}}\right )+4\cdot 3^{k-1} n^{-1/10} \\
    &+(5/2)\Psi \left( \frac{\eta+\mu_{k-1,1}+\lambda_{k}n_{k}}{(\sigma_{k-1,1}^2+n_{k}(1-\bar \lambda_{k}^2))^{1/2}}\right)+(15/4)(1+n^{-1/10})^{k-1} \Psi \left( \frac{\eta+\mu_{k-1,2}+\lambda_{k}n_{k}}{(\sigma_{k-1,2}^2+n_{k}(1-\bar \lambda_{k}^2))^{1/2}}\right)\\
    &\leq 2\cdot 3^{k}\left( n^{-1/10} +\Psi \left( \frac{\eta+\mu_{k,1}}{\sigma_{k,1}}\right)+(3/2)(1+n^{-1/10})^{k} \Psi \left( \frac{\eta+\mu_{k,2}}{\sigma_{k,2}}\right)\right),
\end{align*}
with probability larger than $1-k\exp(-n^{1/75})$. Note that the last inequality follows from the fact that the second last term is much smaller that $\Psi((\eta-T_{k-1})/n_{k}^{1/2})$. Therefore, the lemma holds. 
\end{proof}

\subsubsection{Proof of Lemma \ref{l:together2}}\label{sss:interp}
We firstly present similar lemmas as in subsection \ref{ss:inductSBP}. Recall our definition of the tree structure. For simplicity, we write $\vec h = (h_0,\cdots,h_k)$. Then for any $0\leq h_k\leq D_k$, we have the following
\begin{align*}
    S^{(r)}(v,h_0,\cdots,h_k)(k:k)&=\sum_{j\in \cC_{k,1}} G_{r,j}T(v,h_0,\cdots,h_{k-1},0)_j+\sum_{j\in \cC_{k,2}} G_{r,j}T(v,h_0,\cdots,h_{k-1},D_{k})_j\\
    &:=S_1^{(r)}(v,\vec h )(k:k)+S_2^{(r)}(v,\vec h)(k:k),
\end{align*}
for some partition of sets $\cC_k$ into $\cC_{k,1}$ and $\cC_{k,2}$. Define $n_{k,1}=|\cC_{k,1}|$ and $n_{k,2}=|\cC_{k,2}|$. We use $\cR_k$ to denote the rows involved in computation in round $k$ for $(h_0,\cdots,h_{k-1})$ and $\cR_k'$ for $(h_0,\cdots,h_{k-1})^{+1}$. Further, we use $\vec h_1$ to denote the vector $(h_1,\cdots,h_{k-1})$ and $\vec h_2$ to denote the vector $(h_1,\cdots,h_{k-1})^{+1}$. 
\begin{lem}\label{l:roundisum2}
For any $1<k< R$, write $t=m_k$ and assume $\{r_1,\cdots,r_t\}$ are the $t$ elements of $\cR_k$. For $\eta \in \bR$, we define 
\begin{align*}
    &\gamma_i:=I(\sgn(S^{(r_i)}(v,\vec h_1)(0:k-1))S_1^{(r_i)}(v,\vec h)(k:k)\geq \eta),\\
    &\bar \gamma_i:=I(\sgn(S^{(r_i)}(v,\vec h_1)(0:k-1))S_1^{(r_i)}(v,\vec h)(k:k)\leq \eta).
\end{align*}
Further, define
\begin{align*}
    &q:=\exp(-m_k^{1/6})+\frac{9}{8}\Psi\left( \frac{\eta+ \lambda_k n_{k,1}}{ (n_{k,1}(1-\lambda_k^2))^{1/2} }\right),\\
    &\bar q:=\exp(-m_k^{1/6})+\frac{9}{8}\Psi\left( \frac{-\eta- \bar \lambda_k n_{k,1}}{ (n_{k,1}(1- \bar\lambda_k^2))^{1/2} }\right).
\end{align*}
 Then we have for $n$ large enough,
\begin{align*}
    \bP \left( \sum_{i=1}^t \gamma_i \geq t^{3/5}+(1+t^{-1/{12}})qt  \right)\leq \exp(-t^{1/{70}}),
\end{align*}
and
\begin{align*}
    \bP \left( \sum_{i=1}^t \bar \gamma_i \geq t^{3/5}+(1+t^{-1/{12}})\bar qt  \right)\leq \exp(-t^{1/{70}}).
\end{align*}
\end{lem}
\begin{proof}
The proof is the same as before.
\end{proof}

Similarly, we have the version for $\cR_k'$.
\begin{lem}\label{l:roundisum20}
For any $1<k< R$, write $t=m_k$ and assume $\{r_1,\cdots,r_t\}$ are the $t$ elements of $\cR_k'$. For $\eta \in \bR$, we define 
\begin{align*}
    &\gamma_i:=I(\sgn(S^{(r_i)}(v,\vec h_2)(0:k-1))S_2^{(r_i)}(v,\vec h)(k:k)\geq \eta),\\
    &\bar \gamma_i:=I(\sgn(S^{(r_i)}(v,\vec h_2)(0:k-1))S_2^{(r_i)}(v,\vec h)(k:k)\leq \eta).
\end{align*}
Further, define
\begin{align*}
    &q:=\exp(-m_k^{1/6})+\frac{9}{8}\Psi\left( \frac{\eta+ \lambda_k n_{k,2}}{ (n_{k,2}(1-\lambda_k^2))^{1/2} }\right),\\
    &\bar q:=\exp(-m_k^{1/6})+\frac{9}{8}\Psi\left( \frac{-\eta- \bar \lambda_k n_{k,2}}{ (n_{k,2}(1- \bar\lambda_k^2))^{1/2} }\right).
\end{align*}
 Then we have for $n$ large enough,
\begin{align*}
    \bP \left( \sum_{i=1}^t \gamma_i \geq t^{3/5}+(1+t^{-1/{12}})qt  \right)\leq \exp(-t^{1/{70}}),
\end{align*}
and
\begin{align*}
    \bP \left( \sum_{i=1}^t \bar \gamma_i \geq t^{3/5}+(1+t^{-1/{12}})\bar qt  \right)\leq \exp(-t^{1/{70}}).
\end{align*}
\end{lem}
\begin{proof}
The proof is the same as before. 
\end{proof}

The two lemmas above together imply the same Lemma \ref{l:kstep} and Lemma \ref{l:kstep2} with $n_k$ replaced by $n_{k,1}$ and $n_{k,2}$ and $S(k:k\mid\cR_k)$ and $S(k:k\mid\cR_k^c)$ replaced by $S_1(v,\vec h)(k:k\mid\cR_k)$, $S_1(v,\vec h)(k:k\mid\cR_k^c)$ and $S_2(v,\vec h)(k:k\mid\cR_k')$, $S_2(v,\vec h)(k:k\mid\cR_k'^c)$ respectively. 

\begin{proof}[Proof of Lemma \ref{l:together2}]
The proof is similar to the proof of Lemma \ref{l:together}.
We define $t=m_{k}$ and write $\{r_1,\cdots,r_t\}$ as the set $\cR_{k}$. For $r \in \cR_{k}$, we can define iid random variables $W_1,\cdots,W_t$ on the same space as $\{S^{(r)}(v,\vec h)(0:k-1)\}_{r\in \cR_{k+1}}$ such that 
\begin{align*}
    W_i\geq |S^{(r_i)}(v,\vec h)(0:k-1)|, \forall i \in [t]
\end{align*}
with probability larger than $1-(k-1)\exp(-n^{1/75})-\exp(-n^{1/72})$. Further, we can define iid random variables $\xi_1,\cdots,\xi_t$ and $\bar \xi_1,\cdots,\bar \xi_t$ such that 
\begin{align*}
    \bar \xi_i \leq \sgn(S^{(r_i)}(v,\vec h_1)(0:k-1))S_1^{(r_i)}(v,\vec h)(k:k)\leq \xi_i
\end{align*}
with probability larger than $1-\exp(-n^{1/73})$. And similarly, if we write $\{r'_1,\cdots,r'_t\}$ as the set $\cR_{k}'$, then we can define iid random variables $\xi'_1,\cdots,\xi'_t$ and $\bar \xi'_1,\cdots,\bar \xi'_t$ such that 
\begin{align*}
    \bar \xi'_i \leq \sgn(S^{(r_i)}(v,\vec h_2)(0:k-1))S_2^{(r_i)}(v,\vec h)(k:k)\leq \xi'_i
\end{align*}
with probability larger than $1-\exp(-n^{1/73})$. Further, we note that for any $r\in \cR_k \cup \cR_k' $ the sign of $S^{(r_i)}(v,\vec h_1)(0:k-1)$ and $S^{(r_i)}(v,\vec h_2)(0:k-1)$ are the same, as $r$ is among the rows with largest absolute values. We thus denote the sign as $\sgn(S^{(r_i)}(v,\vec h)(0:k-1))$ for simplicity. This implies that for any $r_i=r_j'\in \cR_{k+1}\cap \cR_{k+1}'$,
\begin{align*}
    \bar \xi_i+ \bar \xi'_j\leq \sgn(S^{(r_i)}(v,\vec h)(0:k-1))S^{(r_i)}(v,\vec h)(k:k) \leq W_i+\xi_i+\xi_j'.
\end{align*}
Furthermore, for rows $r\in (\cR_{k+1}\cup \cR_{k+1}')^c$, by Lemma \ref{l:induct2} and Lemma \ref{l:induct3}, we have that 
\begin{align*}
    |S^{(r)}(0:k)|\leq T_{k}+2,
\end{align*}
with probability larger than $1-(k-1)\exp(-n^{1/75})$. This implies that for $r\in (\cR_{k}\cup \cR_{k}')^c$,
\begin{align*}
    \bar B_r\leq \sgn(S^{(r)}(v,\vec h)(0:k-1))S^{(r)}(v,\vec h)(k:k) \leq T_{k}+2+ B_r,
\end{align*}
where $B_r$ and $\bar B_r$ are sums of $n_{k}$ iid random variables $\sim \mathrm{Bern(1/2)}$. For $r\in \cR_{k}\backslash \cR_{k}'$, 
\begin{align*}
    \bar \xi_i+\bar B_i'\leq \sgn(S^{(r_i)}(v,\vec h)(0:k-1))S^{(r_i)}(v,\vec h)(k:k) \leq T_{k}+2+\xi_i+ B_i',
\end{align*}
where $B_i$ and $\bar B_i$ are sums of $n_{k,2}$ iid random variables $\sim \mathrm{Bern(1/2)}$. For $r\in \cR_{k}'\backslash \cR_{k}$, 
\begin{align*}
    \bar \xi_i'+\bar B_i''\leq \sgn(S^{(r_i')}(v,\vec h)(0:k-1))S^{(r_i')}(v,\vec h)(k:k) \leq T_{k}+2+\xi_i'+B_i'',
\end{align*}
where $B_i$ and $\bar B_i$ are sums of $n_{k,1}$ iid random variables $\sim \mathrm{Bern(1/2)}$. The rest of the arguments are the same as in the proof of Lemma \ref{l:together}. By Lemma \ref{l:roundisum2}, Lemma \ref{l:roundisum20}, Lemma \ref{l:induct2}, Lemma \ref{l:induct3} and Lemma \ref{l:kstep2}, we have that 
\begin{align*}
    &\bP_2(|S^{(r)}(0:k)|\geq \eta)\\
    &\leq 2\cdot 3^{k}\left( n^{-1/10} +\Psi \left( \frac{\eta+\mu_{k,1}}{\sigma_{k,1}}\right)+(3/2)(1+n^{-1/10})^{k} \Psi \left( \frac{\eta+\mu_{k,2}}{\sigma_{k,2}}\right)\right),
\end{align*}
with probability larger than $1-k\exp(-n^{1/75})$. Therefore, the lemma holds. 
\end{proof}

\subsubsection{Proof of Lemma \ref{l:together3}}\label{sss:interp2}
In this section, we prove the inductive bound for the ABP model. We firstly show the following lemma.
\begin{lem}\label{l:induct4}
For any $1\leq k\leq R$, the right hand side of \eqref{eq:inductbound2} (take $\eta=T_k$) is bounded by $m_{k}/m$ for large enough $m$. More precisely,
\begin{align*}
    3^k\left( n^{-1/10} +\Psi \left( \frac{-T_k+\mu_{k,1}}{\sigma_{k,1}}\right)+(3/2)(1+n^{-1/10})^k \Psi \left( \frac{-T_k+\mu_{k,2}}{\sigma_{k,2}}\right)\right)\leq m_{k+1}/m.
\end{align*}
\end{lem}
\begin{proof}
When $k=1$, the left hand side is bounded from above by
\begin{align*}
    & 3^k \Psi \left( \frac{1 /2}{ \sqrt{C_\kappa}(\alpha \pi /2)^{1/4}}\right)+2 \cdot 3^k \Psi \left( C_\kappa\right)+o_n(1)\leq \Psi(7)+\Psi(7)\leq m_2/m.
\end{align*}
For $k\geq 2$, the left hand side is bounded from above by
\begin{align*}
    & 3^k \Psi \left( \frac{1 /2^k}{ (2\alpha \pi \Psi( k+5))^{1/4}}\right)+2 \cdot 3^k \Psi \left( 2k+C_\kappa\right)+o_n(1).
\end{align*}
Note that as 
\begin{align*}
\frac{1/2^k}{ (2\alpha \pi \Psi( k+5))^{1/4}}> 2k+5,
\end{align*}
and
\begin{align*}
    3^k \Psi \left( 2k+C_\kappa\right)\leq \Psi \left( k+6\right),
\end{align*}
we have that the inequality holds.
\end{proof}

\begin{proof}[Proof of Lemma \ref{l:together3}]
Similar to the previous subsection, we firstly present similar lemmas as in subsection \ref{ss:inductSBP}. Recall our definition of the tree structure. For simplicity, we write $\vec h = (h_0,\cdots,h_k)$. Then for any $0\leq h_k\leq D_k$, we have the following
\begin{align*}
    S^{(r)}(v,h_0,\cdots,h_k)(k:k)&=\sum_{j\in \cC_{k,1}} G_{r,j}T(v,h_0,\cdots,h_{k-1},0)_j+\sum_{j\in \cC_{k,2}} G_{r,j}T(v,h_0,\cdots,h_{k-1},D_{k})_j\\
    &:=S_1^{(r)}(v,\vec h )(k:k)+S_2^{(r)}(v,\vec h)(k:k),
\end{align*}
for some partition of sets $\cC_k$ into $\cC_{k,1}$ and $\cC_{k,2}$. Define $n_{k,1}=|\cC_{k,1}|$ and $n_{k,2}=|\cC_{k,2}|$. We use $\cR_k$ to denote the rows involved in computation in round $k$ for $(h_0,\cdots,h_{k-1})$ and $\cR_k'$ for $(h_0,\cdots,h_{k-1})^{+1}$. Further, we use $\vec h_1$ to denote the vector $(h_1,\cdots,h_{k-1})$ and $\vec h_2$ to denote the vector $(h_1,\cdots,h_{k-1})^{+1}$. 

The proof is similar to the proof of Lemma \ref{l:together}.
We define $t=m_{k}$ and write $\{r_1,\cdots,r_t\}$ as the set $\cR_{k}$. For $r \in \cR_{k}$, we can define iid random variables $W_1,\cdots,W_t$ on the same space as $\{S^{(r)}(v,\vec h)(0:k-1)\}_{r\in \cR_{k}}$ such that 
\begin{align*}
    W_i\leq S^{(r_i)}(v,\vec h)(0:k-1), \forall i \in [t]
\end{align*}
with probability larger than $1-(k-1)\exp(-n^{1/75})-\exp(-n^{1/72})$. Further, we can define iid random variables $\xi_1,\cdots,\xi_t$ and $\bar \xi_1,\cdots,\bar \xi_t$ such that 
\begin{align*}
    \xi_i \leq S_1^{(r_i)}(v,\vec h)(k:k)
\end{align*}
with probability larger than $1-\exp(-n^{1/73})$. And similarly, if we write $\{r'_1,\cdots,r'_t\}$ as the set $\cR_{k}'$, then we can define iid random variables $\xi'_1,\cdots,\xi'_t$ such that 
\begin{align*}
    \xi'_i \leq S_2^{(r_i)}(v,\vec h)(k:k)
\end{align*}
with probability larger than $1-\exp(-n^{1/73})$. This implies that for any $r_i=r_j'\in \cR_{k}\cap \cR_{k}'$,
\begin{align*}
    W_i+\xi_i+\xi_j'\leq S^{(r_i)}(v,\vec h)(k:k).
\end{align*}
Furthermore, for rows $r\in (\cR_{k}\cup \cR_{k}')^c$, by Lemma \ref{l:induct4}, we have that 
\begin{align*}
    S^{(r)}(0:k)\geq T_{k}-2,
\end{align*}
with probability larger than $1-(k-1)\exp(-n^{1/75})$. This implies that for $r\in (\cR_{k}\cup \cR_{k}')^c$,
\begin{align*}
    B_r+T_{k}-2\leq S^{(r)}(v,\vec h)(k:k),
\end{align*}
where $B_r$ is a sum of $n_{k}$ iid random variables $\sim \mathrm{Bern(1/2)}$. For $r\in \cR_{k}\backslash \cR_{k}'$, 
\begin{align*}
    \xi_i+B_i'+T_{k}-2\leq S^{(r_i)}(v,\vec h)(k:k),
\end{align*}
where $B_i$ is a sum of $n_{k,2}$ iid random variables $\sim \mathrm{Bern(1/2)}$. For $r\in \cR_{k}'\backslash \cR_{k}$, 
\begin{align*}
     \xi_i'+B_i''+T_{k}-2\leq S^{(r_i')}(v,\vec h)(k:k),
\end{align*}
where $B_i$ is a sum of $n_{k,1}$ iid random variables $\sim \mathrm{Bern(1/2)}$. The rest of the arguments are the same as in the proof of Lemma \ref{l:together}. By Lemma \ref{l:stat2sam} and \ref{l:combine}, we can combine the above and have that 
\begin{align*}
    &\bP_2(S^{(r)}(0:k)\leq \eta)\\
    &\leq 3^{k}\left( n^{-1/10} +\Psi \left( \frac{-\eta+\mu_{k,1}}{\sigma_{k,1}}\right)+(3/2)(1+n^{-1/10})^{k} \Psi \left( \frac{-\eta+\mu_{k,2}}{\sigma_{k,2}}\right)\right),
\end{align*}
with probability larger than $1-k\exp(-n^{1/75})$. Therefore, the lemma holds. 
\end{proof}

\section{Proof of Theorem \ref{t:localcluster}}

\subsection{Proof of Lemma \ref{l:connect2}}
The proof is similar to the proof of Lemma \ref{l:together2}. We only need to check step $\sL$. We start with a lemma on the $\kappa'$ solutions. 
\begin{lem}\label{l:bridge}
In the SBP model, if $X$ is a $\kappa'$-solution, for any row $r$, we have that 
\begin{align*}
    &\bP(|S^{(r)}([\ell])|\geq \eta)\leq 2\exp(-n^{-1/5})+\frac{21}{10}\Psi\left(\frac{\eta-\kappa'\ell/\sqrt{n}}{\sqrt{n-\ell }}\right).
\end{align*}
In the ABP model, if $X$ is a $\kappa'$-solution, for any row $r$, we have that 
\begin{align*}
    &\bP(S^{(r)}([\ell])\leq \eta)\leq \exp(-n^{-1/5})+\frac{21}{20}\Psi\left(\frac{-\eta+\kappa'\ell/\sqrt{n}}{\sqrt{n-\ell }}\right).
\end{align*}
\end{lem}
\begin{proof}
The proof is similar to the proof of Lemma \ref{l:binomial0}.
\end{proof}

Recall that we defined $\cA_{0:\sL}(X([\ell]))$ as $X_1$ and $\cA_{0:\sL}(X([\ell+1]))$ as $X_2$. Note that by definition, $X_1$ and $X_2$ can only differ by at most one entry. We firstly show the following lemma on the first step. 
\begin{lem}\label{l:firstinduct1}
In the SBP model, for any $\eta>T_\sL$, any $\kappa'$-solution $X$ and corresponding $X_1$ and $X_2$, we have
\begin{align*}
    \bP_2[|G(0:\sL)X_1|\geq \eta]\leq 2\cdot 3^\sL\left( n^{-1/10} +\Psi \left( \frac{\eta+\mu_{\sL,1}}{\sigma_{\sL,1}}\right)+(3/2)(1+n^{-1/10})^\sL \Psi \left( \frac{\eta+\mu_{\sL,2}}{\sigma_{\sL,2}}\right)\right),
\end{align*}
with probability larger than $1-\exp(-n^{1/80})$. Here, 
\begin{align*}
    \mu_{k,1}=-T_{k-1}, \quad \sigma_{k,1}^2=n_k, \quad \mu_{k,2}=\sum_{i=1}^k n_i\sqrt{2/\pi m_i}, \quad \sigma_{k,2}^2=\sum_{i=0}^k n_i.
\end{align*}
The same holds for $X_2$.
\end{lem}
\begin{proof}
Let $t=m_{\sL}$ and write $\{r_1,\cdots,r_t\}$ as the set $\cR_{\sL}$. 
By Lemma \ref{l:bridge}, we have
\begin{align*}
    \bP(|S^{(r)}([\ell])|\geq \eta)\leq 2\exp(-n^{-1/5})+\frac{21}{10}\Psi\left(\frac{\eta-\kappa'\ell/\sqrt{n}}{\sqrt{n-\ell }}\right).
\end{align*}
Similar as before, we can define $W_1,\cdots,W_t$ on the same space as $\{|S^{(r)}([\ell])|\}_{r\in \cR_\sL}$ with a scaled distribution such that 
\begin{align*}
    W_i\leq |S^{(r)}([\ell])|, \forall i\in [t],
\end{align*}
with probability larger than $1-\exp(-n^{1/73})$. For simplicity, we write $L=\sum_{s=0}^\sL n_s-\ell$ and further define
\begin{align*}
    \check S^{(r)}(\sL:\sL)=\sum_{j=\ell+1}^{\ell+L} G(r,j)X_j.
\end{align*}
For $r_i\in \cR_{\sL}$, we can define iid random variables $\xi_1,\cdots,\xi_t$ and $\bar \xi_1,\cdots,\bar \xi_t$ with cdf
\begin{align*}
    &1-\min \left\{ 1, (5/4)m_{\sL}^{-2/5}+(5/4)\Psi\left ( \frac{\eta+ \lambda_{\sL} L}{(L(1- \lambda_{\sL}^2))^{1/2}}\right )\right\},\\
    &\min \left\{ 1, (5/4)m_{\sL}^{-2/5}+(5/4)\Psi\left ( \frac{-\eta-\bar \lambda_{\sL} L}{(L(1-\bar \lambda_{\sL}^2))^{1/2}}\right)\right\},
\end{align*}
such that with probability larger than $1-\exp(-n^{1/73})$, we have
\begin{align*}
    \bar \xi_i\leq \sgn(S^{(r_i)}(0:\sL-1))\check S^{(r_i)}(\sL:\sL)\leq \xi_i.
\end{align*}
By Lemma \ref{l:combine} and Lemma \ref{l:stat2sam}, we can combine the above and bound the empirical distribution for $r\in \cR_\sL$. We have with probability larger than $1-\exp(-n^{1/74})$,
\begin{align*}
    &\bP_2(\sgn(S(0:\sL-1\mid \cR_\sL)) S(0:\sL\mid \cR_\sL)\geq \eta)\\&\leq \frac{m}{m_\sL}\left(n^{-1/10}+\frac{21}{16}\Psi\left(\frac{\eta+\lambda_\sL L-\kappa' \ell/\sqrt{n}}{(n-\ell+L(1-\lambda_\sL^2))^{1/2}}\right)\right).
\end{align*}
Similarly, we have
\begin{align*}
    &\bP_2(\sgn(S(0:\sL-1\mid \cR_\sL)) S(0:\sL\mid \cR_\sL)\leq \eta)\\&\leq \frac{m}{m_\sL}\left(n^{-1/10}+\frac{5}{4}\Psi\left(\frac{-\eta-\bar \lambda_\sL L}{(L(1-\bar \lambda_\sL^2))^{1/2}}\right)\right).
\end{align*}
Note that as $0\leq L\leq n_\sL$, $\bar \lambda_\sL\geq \sqrt{2/\pi m_\sL}$, we have for any $\eta>T_{\sL}$,
\begin{align*}
    \frac{\eta-\bar \lambda_\sL L}{(L(1-\bar \lambda_\sL^2))^{1/2}}\geq \frac{\eta+\mu_{\sL,1}}{\sigma_{\sL,1}},\quad \quad 
\end{align*}
Further as we have $\varepsilon_0=(\kappa-\kappa')/2$ and $\varepsilon_0/\sqrt{5d}\geq C_\kappa-\log(d)\kappa/2$, then for any $\eta>T_{\sL}$
\begin{align*}
    \frac{\eta+\lambda_\sL L-\kappa' \ell/\sqrt{n}}{(n-\ell+L(1-\lambda_\sL^2))^{1/2}} \geq \frac{\eta-\kappa' \ell/\sqrt{n}}{(2(n-\ell))^{1/2}}\geq \frac{\eta-T_\sL}{(2(n-\ell))^{1/2}}+\frac{C_\kappa+(\sL+2)\kappa/2}{\sqrt{n}} \geq\frac{\eta+\mu_{\sL,2}}{\sigma_{\sL,2}}.
\end{align*}
These imply that 
\begin{align*}
   &\bP_2[G(0:\sL)X_1\geq \eta \mid \cR_\sL]\\
   &\leq \frac{m}{m_\sL} \left(2n^{-1/10}+\frac{5}{4}\Psi\left(\frac{\eta-\bar \lambda_\sL L}{(L(1-\bar \lambda_\sL^2))^{1/2}}\right)+\frac{21}{16}\Psi\left(\frac{\eta+\lambda_\sL L-\kappa' \ell/\sqrt{n}}{(n-\ell+L(1-\lambda_\sL^2))^{1/2}}\right)\right)\\
   &\leq 3^\sL\left( n^{-1/10} +\Psi \left( \frac{\eta+\mu_{\sL,1}}{\sigma_{\sL,1}}\right)+(3/2)(1+n^{-1/10})^\sL \Psi \left( \frac{\eta+\mu_{\sL,2}}{\sigma_{\sL,2}}\right)\right).
\end{align*}
For rows $r\in \cR_\sL^c$, the argument is similar. We have that with probability larger than $1-\exp(-n^{1/74})$,
\begin{align*}
    &\bP_2(\sgn(S(0:\sL-1\mid \cR_\sL)) S(0:\sL\mid \cR_\sL)\geq \eta)\\&\leq \frac{m}{m-m_\sL}\left(n^{-1/10}+\frac{21}{16}\Psi\left(\frac{\eta-\kappa' \ell/\sqrt{n}}{(n-\ell+L)^{1/2}}\right)\right).
\end{align*}
Similarly, we have
\begin{align*}
    &\bP_2(\sgn(S(0:\sL-1\mid \cR_\sL)) S(0:\sL\mid \cR_\sL)\leq \eta)\\&\leq \frac{m}{m-m_\sL}\left(n^{-1/10}+\frac{5}{4}\Psi\left(\frac{-\eta}{L^{1/2}}\right)\right).
\end{align*}
By the same set of inequalities, we have
\begin{align*}
   &\bP_2[G(0:\sL)X_1\geq \eta \mid \cR_\sL^c]\\
   &\leq 3^\sL\left( n^{-1/10} +\Psi \left( \frac{\eta+\mu_{\sL,1}}{\sigma_{\sL,1}}\right)+(3/2)(1+n^{-1/10})^\sL \Psi \left( \frac{\eta+\mu_{\sL,2}}{\sigma_{\sL,2}}\right)\right).
\end{align*}
Together with the bound on $\cR_\sL$, we have the lemma.
\end{proof}

\begin{lem}\label{l:firstinduct1a}
In the ABP model, for any $\eta<T_\sL$, any $\kappa'$-solution $X$ and corresponding $X_1$ and $X_2$, we have
\begin{align*}
    \bP_2[G(0:\sL)X_1\leq \eta]\leq  3^\sL\left( n^{-1/10} +\Psi \left( \frac{-\eta+\mu_{\sL,1}}{\sigma_{\sL,1}}\right)+(3/2)(1+n^{-1/10})^\sL \Psi \left( \frac{-\eta+\mu_{\sL,2}}{\sigma_{\sL,2}}\right)\right),
\end{align*}
with probability larger than $1-\exp(-n^{1/80})$. Here, 
\begin{align*}
    \mu_{k,1}=-T_{k-1}, \quad \sigma_{k,1}^2=n_k, \quad \mu_{k,2}=\sum_{i=1}^k n_i\sqrt{2/\pi m_i}, \quad \sigma_{k,2}^2=\sum_{i=0}^k n_i.
\end{align*}
The same holds for $X_2$.
\end{lem}
\begin{proof}
The argument is very similar to the SBP model. Let $t=m_{\sL}$ and write $\{r_1,\cdots,r_t\}$ as the set $\cR_{\sL}$. 
By Lemma \ref{l:bridge}, we have
\begin{align*}
    \bP(S^{(r)}([\ell])\leq \eta)\leq \exp(-n^{-1/5})+\frac{21}{20}\Psi\left(\frac{-\eta+\kappa'\ell/\sqrt{n}}{\sqrt{n-\ell}}\right).
\end{align*}
Similar as before, we can define $W_1,\cdots,W_t$ on the same space as $\{S^{(r)}([\ell])\}_{r\in \cR_\sL}$ with a scaled distribution such that 
\begin{align*}
    W_i\leq S^{(r)}([\ell]), \forall i\in [t],
\end{align*}
with probability larger than $1-\exp(-n^{1/73})$. And we can define iid random variables $\xi_1,\cdots,\xi_t$ with cdf
\begin{align*}
    &\min \left\{ 1, (5/4)m_{\sL}^{-2/5}+(5/4)\Psi\left ( \frac{-\eta+\bar \lambda_{\sL} L}{(L(1-\bar \lambda_{\sL}^2))^{1/2}}\right)\right\},
\end{align*}
such that with probability larger than $1-\exp(-n^{1/73})$, we have
\begin{align*}
    \xi_i\leq \check S^{(r_i)}(\sL:\sL).
\end{align*}
By Lemma \ref{l:combine} and Lemma \ref{l:stat2sam}, we can combine the above and bound the empirical distribution for $r\in \cR_\sL$. We have with probability larger than $1-\exp(-n^{1/74})$,
\begin{align*}
    &\bP_2(S(0:\sL\mid \cR_\sL)\leq \eta)\\&\leq \frac{m}{m_\sL}\left(n^{-1/10}+\frac{21}{16}\Psi\left(\frac{-\eta+\lambda_\sL L+\kappa' \ell/\sqrt{n}}{n-\ell+L(1-\lambda_\sL^2))^{1/2}}\right)\right).
\end{align*}
Note that as we have $\varepsilon_0=(\kappa'-\kappa)/2$ and $\varepsilon_0/\sqrt{5d}\geq C_\kappa-2\log(d)$, then for any $\eta<T_{\sL}$,
\begin{align*}
    \frac{-\eta+\lambda_\sL L+\kappa' \ell/\sqrt{n}}{(n-\ell+L(1-\lambda_\sL^2))^{1/2}} \geq \frac{-\eta+\kappa' \ell/\sqrt{n}}{(2(n-\ell))^{1/2}}\geq \frac{-\eta+T_\sL}{(2(n-\ell))^{1/2}}+\frac{C_\kappa+2(\sL+2)}{\sqrt{n}} \geq\frac{-\eta+\mu_{\sL,2}}{\sigma_{\sL,2}}.
\end{align*}
These imply that 
\begin{align*}
   &\bP_2[G(0:\sL)X_1\leq \eta \mid \cR_\sL]\\
   &\leq \frac{m}{m_\sL} \left(n^{-1/10}+\frac{21}{16}\Psi\left(\frac{-\eta+\lambda_\sL L+\kappa' \ell/\sqrt{n}}{(n-\ell+L(1-\lambda_\sL^2))^{1/2}}\right)\right)\\
   &\leq 3^\sL\left( n^{-1/10} +\Psi \left( \frac{-\eta+\mu_{\sL,1}}{\sigma_{\sL,1}}\right)+(3/2)(1+n^{-1/10})^\sL \Psi \left( \frac{-\eta+\mu_{\sL,2}}{\sigma_{\sL,2}}\right)\right).
\end{align*}
For rows $r\in \cR_\sL^c$, the argument is similar. We have that with probability larger than $1-\exp(-n^{1/74})$,
\begin{align*}
    &\bP_2(\sgn(S(0:\sL-1\mid \cR_\sL)) S(0:\sL\mid \cR_\sL)\leq \eta)\\&\leq \frac{m}{m-m_\sL}\left(n^{-1/10}+\frac{21}{16}\Psi\left(\frac{-\eta+\kappa' \ell/\sqrt{n}}{(n-\ell+L)^{1/2}}\right)\right).
\end{align*}
By the same set of inequalities, we have
\begin{align*}
   &\bP_2[G(0:\sL)X_1\leq \eta \mid \cR_\sL^c]\\
   &\leq 3^\sL\left( n^{-1/10} +\Psi \left( \frac{-\eta+\mu_{\sL,1}}{\sigma_{\sL,1}}\right)+(3/2)(1+n^{-1/10})^L \Psi \left( \frac{-\eta+\mu_{\sL,2}}{\sigma_{\sL,2}}\right)\right).
\end{align*}
Together with the bound on $\cR_\sL$, we have the lemma.
\end{proof}

\begin{proof}[Proof of Lemma \ref{l:connect2}]
Our proof is inductive. The induction basis follows from Lemma \ref{l:firstinduct1} and Lemma \ref{l:firstinduct1a} and the induction steps follow from the same proof as in Section \ref{sss:interp}. 
\end{proof}

\subsection{Proof of Lemma \ref{l:connect3}}
We only need to check step $\sL$. We start with a lemma on the $\kappa'$-solutions. Recall that we defined $\sF(X,\ell)([n-\sd])$ as $X_1$ and $\sF(X,\ell+1)([n-\sd])$ as $X_2$.
\begin{lem}\label{l:bridge2}
In the SBP model, if $X$ is a $\kappa'$-solution, for any row $r$, we have that 
\begin{align*}
    &\bP(|G([n-\sd])X_1|\geq \eta)\leq 2\exp(-n^{-1/5})+\frac{21}{10}\Psi\left(\frac{\eta-\kappa'(n-\sd-2\ell)/\sqrt{n}}{\sqrt{\sd }}\right)
\end{align*}
In the ABP model, if $X$ is a $\kappa'$-solution, for any row $r$, we have that 
\begin{align*}
    &\bP(G([n-\sd])X_1\leq \eta)\leq \exp(-n^{-1/5})+\frac{21}{20}\Psi\left(\frac{-\eta+\kappa'(n-\sd-2\ell)/\sqrt{n}}{\sqrt{\sd}}\right)
\end{align*}
\end{lem}
\begin{proof}
The proof is similar to the proof of Lemma \ref{l:binomial0}.
\end{proof}
We firstly show the following lemma on the $\sL$ step. 
\begin{lem}\label{l:firstinduct10}
In the SBP model, for any $\eta>T_\sL$, any $\kappa'$-solution $X$ and corresponding $X_1$ and $X_2$, we have
\begin{align*}
    \bP_2[|G(0:\sL)X_1|\geq \eta]\leq 2\cdot 3^\sL\left( n^{-1/10} +\Psi \left( \frac{\eta+\mu_{\sL,1}}{\sigma_{\sL,1}}\right)+(3/2)(1+n^{-1/10})^\sL \Psi \left( \frac{\eta+\mu_{\sL,2}}{\sigma_{\sL,2}}\right)\right),
\end{align*}
with probability larger than $1-\exp(-n^{1/80})$. Here, 
\begin{align*}
    \mu_{k,1}=-T_{k-1}, \quad \sigma_{k,1}^2=n_k, \quad \mu_{k,2}=\sum_{i=1}^k n_i\sqrt{2/\pi m_i}, \quad \sigma_{k,2}^2=\sum_{i=0}^k n_i.
\end{align*}
In the ABP model, for any $\eta<T_\sL$, we have
\begin{align*}
    \bP_2[G(0:\sL)X_1\leq \eta]\leq 3^\sL\left( n^{-1/10} +\Psi \left( \frac{-\eta+\mu_{\sL,1}}{\sigma_{\sL,1}}\right)+(3/2)(1+n^{-1/10})^\sL \Psi \left( \frac{-\eta+\mu_{\sL,2}}{\sigma_{\sL,2}}\right)\right),
\end{align*}
with probability larger than $1-\exp(-n^{1/80})$. Here, 
\begin{align*}
    \mu_{k,1}=T_{k-1}, \quad \sigma_{k,1}^2=n_k, \quad \mu_{k,2}=\sum_{i=1}^k n_i\sqrt{2/\pi m_i}, \quad \sigma_{k,2}^2=\sum_{i=0}^k n_i.
\end{align*}
The same holds for $X_2$.
\end{lem}
\begin{proof}
The proof is very similar to the previous section. In the SBP model, we can bound the row sums by sums of $W_i$ and $\xi_i$ and have that for 
\begin{align*}
   &\bP_2[|G(0:\sL)\cA_{0:\sL}(X_1)|\geq \eta]\\
   &\leq 2n^{-1/10}+3\Psi\left(\frac{\eta-\bar \lambda_\sL L}{(L(1-\bar \lambda_\sL^2))^{1/2}}\right)+3\Psi\left(\frac{\eta+\lambda_\sL L-\kappa' (n-\sd-2\ell)/\sqrt{n}}{(\sd+L(1-\lambda_\sL^2))^{1/2}}\right).
\end{align*}
It remains to check that the right hand side is bounded. Note that as $0\leq L\leq n_\sL$, $\bar \lambda_\sL\geq \sqrt{2/\pi m_\sL}$, we have for any $\eta>T_{\sL}$,
\begin{align*}
    \frac{\eta-\bar \lambda_\sL L}{(L(1-\bar \lambda_\sL^2))^{1/2}}\geq \frac{\eta+\mu_{\sL,1}}{\sigma_{\sL,1}},\quad \quad 
\end{align*}
Further as we have $\varepsilon_0=(\kappa-\kappa')/2$ and $\varepsilon_0/\sqrt{5d}\geq C_\kappa-\log(d)\kappa/2$, then for any $\eta>T_{\sL}$
\begin{align*}
    \frac{\eta+\lambda_\sL L-\kappa' (n-\sd-2\ell)/\sqrt{n}}{(\sd+L(1-\lambda_\sL^2))^{1/2}} \geq \frac{\eta-\kappa' (n-\sd-2\ell)/\sqrt{n}}{(2\sd)^{1/2}}\geq \frac{\eta-T_\sL}{(2\sd)^{1/2}}+\frac{C_\kappa+(\sL+2)\kappa/2}{\sqrt{n}} \geq\frac{\eta+\mu_{\sL,2}}{\sigma_{\sL,2}}.
\end{align*}
Therefore, we can replace the fractions in the inequality and thus prove the result for the SBP case. For ABP, the argument is similar. We have that
\begin{align*}
   &\bP_2[G(0:\sL)\cA_{0:\sL}(X_1)\leq \eta]\\
   &\leq 2n^{-1/10}+3\Psi\left(\frac{-\eta+\lambda_\sL L+\kappa' (n-\sd-2\ell)/\sqrt{n}}{(\sd+L(1-\lambda_\sL^2))^{1/2}}\right).
\end{align*}
It remains to check that the right hand side is bounded. Note that we have $\varepsilon_0=(\kappa-\kappa')/2$ and $\varepsilon_0/\sqrt{5d}\geq C_\kappa-2\log(d)$, then for any $\eta<T_{\sL}$
\begin{align*}
    \frac{-\eta+\lambda_\sL L+\kappa' (n-\sd-2\ell)/\sqrt{n}}{(\sd+L(1-\lambda_\sL^2))^{1/2}} &\geq \frac{-\eta+\kappa' (n-\sd-2\ell)/\sqrt{n}}{(2\sd)^{1/2}}\\&\geq \frac{-\eta+T_\sL}{(2\sd)^{1/2}}+\frac{C_\kappa+2(\sL+2)}{\sqrt{n}} \geq\frac{-\eta+\mu_{\sL,2}}{\sigma_{\sL,2}}.
\end{align*}
Therefore, we can replace the fractions in the inequality and thus prove the result for the ABP case.
\end{proof}

\begin{proof}[Proof of Lemma \ref{l:connect3}]
Our proof is inductive. The induction basis follows from Lemma \ref{l:firstinduct10} and the induction steps follow from the same proof as in Section \ref{sss:interp}. Combining the arguments, we have Lemma \ref{l:connect3}. Moreover, together with Lemma \ref{l:connect2}, we also establish Theorem \ref{t:localcluster}.
\end{proof}

\subsection{Proof of Corollaries}
\begin{proof}[Proof of Corollary \ref{c:1}]
In the SBP model, for any $0<\alpha<\tilde \alpha_c(\kappa)$, there exists $0<\kappa'<\kappa$ such that $\alpha<\tilde \alpha_c(\kappa')$. This implies the existence of $\kappa'$-solutions. By Theorem \ref{t:localcluster}, there is $d>0$ such that there exists a cluster with diameter at least $dn$ with high probability.
\end{proof}

\begin{proof}[Proof of Corollary \ref{c:2}]
In the ABP model, under Assumption \ref{a:cont}, for any $0<\alpha< \alpha_c(\kappa)$, there exists $\kappa'>\kappa$ such that $\alpha< \alpha_c(\kappa')$. This implies the existence of $\kappa'$-solutions. By Theorem \ref{t:localcluster}, there is $d>0$ such that there exists a cluster with diameter at least $dn$ with high probability.
\end{proof}
\printbibliography
\end{document}